\documentclass[preprint,11pt,bbm]{imsart}
\RequirePackage[OT1]{fontenc}
\RequirePackage[numbers]{natbib}
\RequirePackage[colorlinks,citecolor=blue,urlcolor=blue]{hyperref}
\usepackage{amsmath}
\usepackage{amsfonts, mathrsfs}
\usepackage{amssymb}
\usepackage{verbatim}
\usepackage{graphicx}
\usepackage{color}
\usepackage{natbib}
\usepackage{url}
\usepackage{enumerate}
\usepackage{bm}
\usepackage{subcaption}
\usepackage{xcolor}

\definecolor{applegreen}{rgb}{0.55, 0.71, 0.0}

\newcommand{\bfi}{\boldsymbol{i}}
\newcommand{\bfj}{\boldsymbol{j}}

\newcommand{\bfJ}{\boldsymbol{J}}

\numberwithin{equation}{section}
\newtheorem{theo}{Theorem}[section]
\newtheorem{cor}[theo]{Corollary}

\newtheorem{lemma}[theo]{Lemma}
\newtheorem{example}[theo]{Example}
\newtheorem{assum}[theo]{Assumption}

\newtheorem{remark}[theo]{Remark}
\newenvironment{proof}[1][Proof]{\textbf{#1.} }{\ \rule{0.5em}{0.5em}}

\newcommand{\var}{{\rm Var} \mspace{1mu}}

\allowdisplaybreaks

\begin{document}

\begin{frontmatter}
\title{Importance sampling for maxima on trees
}
\runtitle{Importance sampling for maxima on trees}

\begin{aug}
\author{\fnms{Bojan} \snm{Basrak}\thanksref{t1}\ead[label=e1]{bbasrak@math.hr}},
\author{\fnms{Michael} \snm{Conroy}\thanksref{t2}
\ead[label=e2]{mconroy@live.unc.edu}},
\author{\fnms{Mariana} \snm{Olvera-Cravioto}\thanksref{t2,t3}\ead[label=e3]{molvera@email.unc.edu}} \\
\and
\author{\fnms{Zbigniew} \snm{Palmowski}\thanksref{t4}
\ead[label=e4]{zbigniew.palmowski@pwr.edu.pl}
}

\thankstext{t1}{This work is in part financed by Croatian Science Foundation and the Swiss National Science Foundation -- grant CSRP 018-01-180549.}
\thankstext{t3}{This work is partially supported by NSF Grant No. NSF CMMI-1537638.}
\thankstext{t4}{This work is partially supported by Polish National Science Centre Grant No. 2018/29/B/ST1/00756, 2019-2022.}
\runauthor{B. Basrak et al.}

\affiliation{University of Zagreb\thanksmark{t1}, University of North Carolina\thanksmark{t3} and Wroc\l aw University of Science and Technology\thanksmark{t4}}

\address{
\thanksmark{t1}Department of Mathematics, Faculty of Science\\
University of Zagreb\\
\printead{e1}
}

\address{
\thanksmark{t2}Department of Statistics and Operations Research\\
University of North Carolina at Chapel Hill\\
\printead{e2}\\
\printead*{e3}
}

\address{\thanksmark{t4}Faculty of Pure and Applied Mathematics, Hugo Steinhaus Center\\
Wroc\l aw University of Science and Technology\\
\printead{e4}
}
\end{aug}

\begin{abstract}
We consider
the distributional fixed-point equation:
\[R \stackrel{\mathcal{D}}{=} Q \vee \left( \bigvee_{i=1}^N C_i R_i \right),
\]
where the $\{R_i\}$ are i.i.d.~copies of $R$, independent of the vector $(Q, N, \{C_i\})$, where $N \in \mathbb{N}$, $Q, \{C_i\} \geq 0$ and $P(Q > 0) > 0$.
By setting $W = \log R$, $X_i = \log C_i$, $Y = \log Q$ it is equivalent to the high-order Lindley equation
\[W \stackrel{\mathcal{D}}{=} \max\left\{ Y, \, \max_{1 \leq i \leq N} (X_i + W_i) \right\}.\]
It is known that under Kesten assumptions, \[P(W > t) \sim H e^{-\alpha t}, \qquad t \to \infty,\]
where $\alpha>0$ solves the Cram\'er-Lundberg equation $E \left[  \sum_{j=1}^N C_i ^\alpha \right] = E\left[ \sum_{i=1}^N e^{\alpha X_i} \right] = 1$.
The main goal of this paper is to provide an explicit representation for $P(W > t)$, which can be directly connected to the underlying weighted branching process where $W$ is constructed and that can be used to construct unbiased and strongly efficient estimators for all $t$. Furthermore, we show how this new representation can be directly analyzed using Alsmeyer's Markov renewal theorem, yielding an alternative representation for the constant $H$. We provide numerical examples illustrating the use of this new algorithm.
\end{abstract}



\end{frontmatter}

\section{Introduction}

The distributional fixed-point equation:
\begin{equation} \label{eq:HighOrderLindley}
W \stackrel{\mathcal{D}}{=} \max\left\{ Y, \, \max_{1 \leq i \leq N} (X_i + W_i) \right\},
\end{equation}
where the $\{W_i\}$ are i.i.d.~copies of $W$, independent of the vector $(Y, N, \{X_i\})$, with $N \in \mathbb{N}$, is known in the literature as the high-order Lindley equation \cite{Biggins_98,Jel_Olv_15,Kar_Kel_Suh_94, Olv_Rui_19}. The special case of $N \equiv 1$ and $Y \equiv 0$, known as the Lindley equation,
\begin{equation} \label{eq:Lindley}
W \stackrel{\mathcal{D}}{=} \max\left\{ 0, X + W \right\},
\end{equation}
 is perhaps one of the best studied recursions in applied probability, since it describes the stationary distribution of the waiting time in a single-server queue fed by a renewal process and having i.i.d.~service times; see  Asmussen~\cite{Asm2003} and Cohen~\cite{cohen-SSQ} for a comprehensive overview. If we replace the zero in \eqref{eq:Lindley} with a random $Y$ we obtain a recursion satisfied by the all-time supremum of a ``perturbed" random walk, where the $Y$ denotes the perturbation.
 This type of distributional recursion was analyzed, for example, in \cite{Araman, Grey, Iksanov}.
 The branching form \eqref{eq:HighOrderLindley} appears in the study of queueing networks with synchronization requirements \cite{Kar_Kel_Suh_94,Olv_Rui_19} and in the analysis of the maximum displacement of a branching random walk \cite{Biggins_98}.

Although Lindley's equation has a unique solution whenever $E[X] < 0$, there is no uniqueness in the branching case, as shown in \cite{Kar_Kel_Suh_94}. As the work in \cite{Biggins_98} shows, the solutions to \eqref{eq:HighOrderLindley} can be constructed using one special solution, known as the endogenous solution \cite{Aldo_Band_05}. The endogenous solution can be explicitly constructed on a structure known as a weighted branching process \cite{Biggins_77,Rosler_93, Jel_Olv_12a}, and other solutions can be obtained by adding different ``terminal'' values to the leaves of a finite tree (see \cite{Biggins_98} and Section~\ref{S.Changeofmeasure} for more details). From an applications point of view (e.g., the models in \cite{Kar_Kel_Suh_94, Olv_Rui_19}), it is usually the special endogenous solution that is of interest.
We recall that if $W = \log R$, $X_i = \log C_i$, $Y = \log Q$, equation \eqref{eq:HighOrderLindley} is related with the random extremal equation
\begin{equation}\label{multeq}
R \stackrel{\mathcal{D}}{=} Q \vee \left( \bigvee_{i=1}^N C_i R_i \right),
\end{equation}
where the $\{R_i\}$ are i.i.d.~copies of $R$, independent of the vector $(Q, N, \{C_i\})$, where $N \in \mathbb{N}\cup \{\infty\}$, $Q, \{C_i\} \geq 0$ and $P(Q > 0) > 0$; throughout the paper we use the notation $x \vee y = \max\{x, y\}$. These types of distributional equations and their simulation have received considerable attention in the recent literature, although most of it has centered around the affine version of the equation studied here; see
\cite{Chen_Lit_Olv_17, Collamore1, Collamore2, Enriquez, Goldie_91,  Guiv, Jel_Olv_12a, Jel_Olv_12b, Jel_Olv_15}.
We refer to the overview on this topic given in \cite{Darek, Iksanov}.

Given both the theoretical and practical importance of the special endogenous solution to \eqref{eq:HighOrderLindley}, the focus of this paper is the study of its asymptotic tail behavior, i.e., $P(W > t)$ for large $t$. The study of this tail distribution in the case of the single-server queue is part of the classical queueing theory literature (see, e.g., Chapter X in \cite{Asm2003}), and it includes both the case when $X$ in \eqref{eq:Lindley} has finite exponential moments and when it is heavy-tailed. Of particular interest to our present work is the Cram\'er-Lundberg asymptotic (see Theorem~5.2 in Chapter XIII of \cite{Asm2003}), which states that
\begin{equation}\label{eq:CramLundAsymptotic}
P(W > t) \sim K e^{-\alpha t}, \qquad t \to \infty,
\end{equation}
where $\alpha > 0$ is the solution to $E[e^{\alpha X} ] = 1$ and $K = \widetilde{E}[ e^{-\alpha B(\infty)}]$ is a constant that can be computed in terms of the overshoot $B(x)$ of level $x$ of the underlying random walk $S_n = X_1 + \dots + X_n$ under a change of measure inducing the probability $\widetilde{P}$.

The corresponding exponential decay of the endogenous solution to \eqref{eq:HighOrderLindley} has been established in \cite{Jel_Olv_15} using implicit renewal theory \cite{Goldie_91,Jel_Olv_12a,Jel_Olv_12b}. Specifically, Theorem~3.4 in \cite{Jel_Olv_15} states that, provided there exists $\alpha > 0$ such that
\[E\left[ \sum_{i=1}^N e^{\alpha X_i} \right] = 1\quad \text{and}\quad 0 < E\left[ \sum_{i=1}^N e^{\alpha X_i} X_i \right] < \infty,\] then
$$P(W > t) \sim H e^{-\alpha t}, \qquad t \to \infty.$$
However, the constant $H$ provided by the theorem is implicitly defined in terms of the $\{W_i\}_{i \geq 1}$ themselves, making its interpretation even less obvious than in the non-branching case. Hence, the main goal of this paper is to provide an alternative representation for $P(W > t)$ yielding: 1) an unbiased and easy to simulate algorithm for $P(W > t)$ for all values of $t$, and 2) an alternative expression for $H$ that better reflects the behavior of the underlying weighted branching random walk leading to the event $\{W > t\}$.   The main tool enabling our first goal is a new interpretation of the measure $ E\left[ \sum_{i=1}^N 1(\log C_i \in dx)\right]$ appearing in the renewal theoretic approaches for establishing the existence of $H$ \cite{Darek, Iksanov_04, Liu_98, Liu_00, Biggins_98, Jel_Olv_12a, Jel_Olv_12b} in terms of a distinguished path, to which we will subsequently apply a change of measure. The second goal, that of obtaining an alternative representation for $H$, is attained by applying the Markov Renewal Theorem from \cite{Alsmeyer_94} to our newly derived representation. The new proposed simulation algorithm yields an unbiased and strongly efficient estimator for the probability $P(W > t)$, much in the spirit of the importance sampling approach provided by Siegmund's algorithm (see Section 2a, Chapter VI in \cite{asmussen2007stochastic}) for the Lindley equation \eqref{eq:Lindley}.
Importance samplers  were also constructed in \cite{Jose1, Collamore1} for the tail distribution of the solution of the affine equation ($N \equiv 1$), in \cite{Siegmund} in the context of sequential analysis, in \cite{GlynnIglehart} for
Markov chains and semi-Markov processes, in \cite{Collamore3} for Markov-modulated walks.
For general review on rare-event simulation we refer the reader to
\cite{Jose2, Bucklew}.

The change of measure we propose is of independent interest, since it differs from the typical one encountered in the weighted branching processes literature.  It is constructed along a random path $\{ {\bfJ}_r \}_{r \geq 0}$ of the underlying weighted branching process, which we refer to as the {\em spine}, and changes its drift while leaving all other paths unchanged.
What is even more interesting is that the likelihood martingale $L_n=\prod_{r=0}^{n-1} D_{{\bfJ}_r}$ used in our approach is  constructed as a product of certain random variables $D_{\bfi}$ along the spine, and it is substantially different from the seminal Biggins-Kyprianou martingale $W_n(\alpha)$ (see e.g. \cite{{Biggins_98}, Kyprianou_00}),
which is constructed along the `width' of the  tree; see Section~\ref{S.HighOrderLindley} for details. Finally, our new change of measure approach also provides important insights into the exponential asymptotics described by the implicit renewal theorem \cite{Goldie_91, Jel_Olv_15}.

The remainder of the paper is organized as follows. Section~\ref{S.WBP} gives a description of a general weighted branching process and Section~\ref{S.HighOrderLindley} explains how the special endogenous solution to \eqref{eq:HighOrderLindley} is constructed. Our main theorem is given in Section~\ref{S.MainResult}, and the importance sampling algorithm is discussed in Section~\ref{S.Simulation}.

\section{The weighted branching process} \label{S.WBP}

We adopt the notation from \cite{Jel_Olv_12b} to define a marked Galton-Watson process. To this end, let $\mathbb{N}_{+} = \{1, 2, 3, \dots\}$ be the set of positive integers and let $U = \bigcup_{k=0}^\infty (\mathbb{N}_+)^k$ be the set of all finite sequences ${\bfi} = (i_1, i_2, \dots, i_n)$, where by convention $\mathbb{N}_+^0 = \{ \emptyset\}$ contains the null sequence $\emptyset$. To ease the exposition, for a sequence ${\bfi} = (i_1, i_2, \dots, i_k) \in U$ we write ${\bfi}|n = (i_1, i_2, \dots, i_n)$, provided $k \geq n$, and  ${\bfi}|0 = \emptyset$ to denote the index truncation at level $n$, $n \geq 0$. Also, for ${\bfi} \in A_1$ we simply use the notation ${\bfi} = i_1$, skipping the parenthesis. Similarly, for ${\bfi} = (i_1, \dots, i_n)$ we will use $({\bfi}, j) = (i_1,\dots, i_n, j)$ to denote the index concatenation operation, and if ${\bfi} = \emptyset$, then write $({\bfi}, j) = j$. Let $|{\bfi}|$ be the length of index ${\bfi}$, i.e., $|{\bfi}| = k$ if ${\bfi} = (i_1, \dots, i_k) \in \mathbb{N}_+^k$. We order $U$ according to a length-lexicographic order $\prec$: ${\bfi} \prec {\bfj}$ if either $|{\bfi}| < |{\bfj}|$, or $|{\bfi}| = |{\bfj}|$ and $i_r = j_r$ for $r = 1, \dots, t-1$, and $i_t < j_t$ for some $t \leq |{\bfi}|$.

To iteratively construct the weighted branching tree $\mathcal{T}$, let $\left\{\bm{\psi}_{\bfi} \right\}_{{\bfi} \in U}$ denote a sequence of i.i.d. random elements in
$\mathbb{N} \times \mathbb{R}^\infty $, where
$\bm{\psi}_{\bfi}= (N_{\bfi}, Q_{\bfi}, C_{({\bfi},1)}, C_{({\bfi}, 2)}, \dots)$.
For simplicity we denote
 $\bm{\psi}= (N,Q, C_{1}, C_{2}, \dots)=\bm{\psi}_\emptyset$ to represent a generic element of the sequence $\{\bm{\psi}_{\bfi}\}$. The random integers $\{ N_{\bfi} \}_{{\bfi} \in U}$ herein define the structure of the tree as follows. Let $A_0 = \left\{ \emptyset \right\}$,
\begin{align}
A_1 &= \{ i \in \mathbb{N}: 1 \leq i \leq N_\emptyset \}, \quad \text{and} \notag \\
A_n &= \{ ({\bfi}, i_n) \in U:  {\bfi} \in A_{n-1}, 1 \leq i_n \leq N_{\bfi} \}, \quad n \geq 2, \label{eq:AnDef}
\end{align}
be the set of individuals in the $n$th generation. Thus to each node ${\bfi}$ in the tree
different from the root we assign the weight $C_{\bfi}$, and a cumulative weight $\Pi_{\bfi}$  computed according to
$$\Pi_{i_1} = C_{i_1}, \qquad \Pi_{(i_1, \dots, i_n)} = C_{(i_1, \dots, i_n)} \Pi_{(i_1, \dots, i_{n-1})}, \qquad n \geq 2, $$
where $\Pi = \Pi_\emptyset \equiv 1$ is the cumulative weight of the root node.
See Figure \ref{F.Tree}.

\begin{figure}[t]
\centering
\begin{picture}(330,110)(10,0)
\put(0,0){\includegraphics[scale = 0.75, bb = 30 560 510 695, clip]{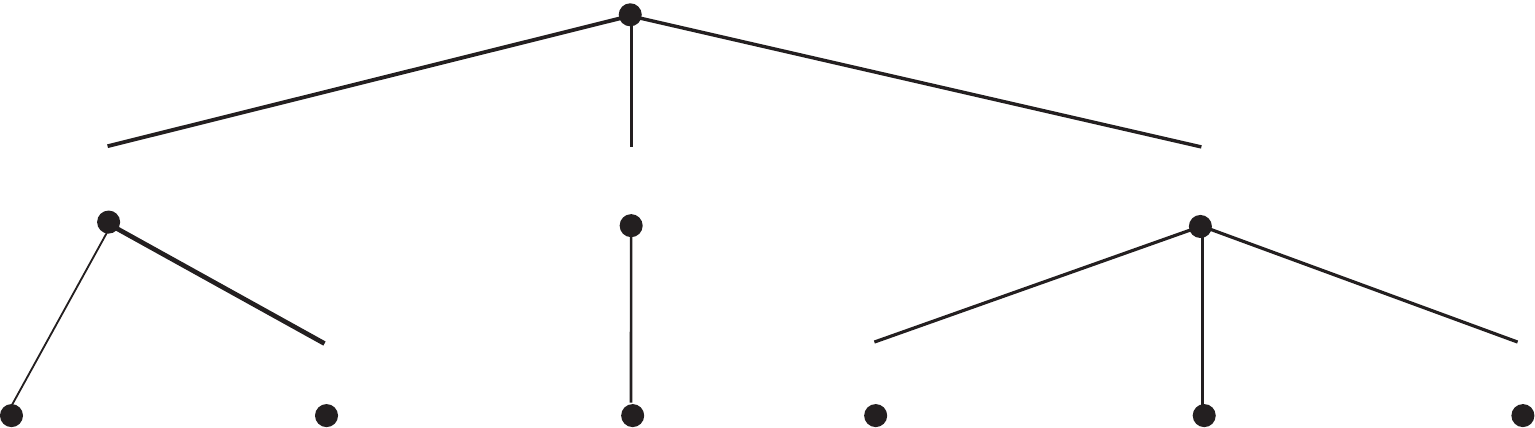}}
\put(27,8){\includegraphics[scale = 0.75]{TreeW}}
\put(150,105){\footnotesize $\Pi = 1$}
\put(30,59){\footnotesize $\Pi_{1} = C_1$}
\put(145,59){\footnotesize $\Pi_{2} = C_2$}
\put(270,59){\footnotesize $\Pi_{3} = C_3$}
\put(0,0){\footnotesize $\Pi_{(1,1)} = C_{(1,1)} C_1$}
\put(62,17){\footnotesize $\Pi_{(1,2)} = C_{(1,2)} C_1$}
\put(129,0){\footnotesize $\Pi_{(2,1)} = C_{(2,1)} C_2$}
\put(182,17){\footnotesize $\Pi_{(3,1)} = C_{(3,1)} C_3$}
\put(252,0){\footnotesize $\Pi_{(3,2)} = C_{(3,2)} C_3$}
\put(318,17){\footnotesize $\Pi_{(3,3)} = C_{(3,3)} C_3$}
\end{picture}
\caption{Weighted branching tree}\label{F.Tree}
\end{figure}

\subsection{The high-order Lindley equation} \label{S.HighOrderLindley}

Consider now the distributional fixed-point equation:
\begin{equation} \label{eq:Maximum}
R \stackrel{\mathcal{D}}{=} Q \vee \left( \bigvee_{i=1}^N C_i R_i \right) ,
\end{equation}
where the $\{R_i\}$ are i.i.d.~copies of $R$, independent of the vector $(Q, N, \{C_i\})$, where $N \in \mathbb{N}$, $Q, \{C_i\} \geq 0$ and $P(Q > 0) > 0$. Recall that by setting $W = \log R$, $X_i = \log C_i$, $Y = \log Q$ we obtain the high-order Lindley equation
$$W \stackrel{\mathcal{D}}{=} \max\left\{ Y, \, \max_{1 \leq i \leq N} (X_i + W_i) \right\}.$$

The random variable
\begin{equation}\label{eq:Endogenous}
R := \bigvee_{{\bfi} \in \mathcal{T}} \Pi_{\bfi} Q_{\bfi}
\end{equation}
is known as the special endogenous solution to \eqref{eq:Maximum}. As mentioned earlier, the high-order Lindley equation has in general multiple solutions, but before we discuss those it is convenient to focus first on the so-called {\em regular} case, which corresponds to the existence of a unique $\alpha > 0$ satisfying $E\left[ \sum_{j=1}^N C_j^\alpha \right] = 1$ and $E\left[ \sum_{j=1}^N C_j^\alpha \log C_j \right] \in (0, \infty)$ (see \cite{Biggins_98}). As the work in \cite{Biggins_98} shows, other solutions to \eqref{eq:Maximum} can be constructed by using ``terminal" values.  More precisely, consider the finite tree $\mathcal{T}^{(n)} = \{ {\bfi} \in \mathcal{T}: |{\bfi}| \leq n\}$, and construct the random variable
$$R_n(B) = \left( \bigvee_{{\bfi} \in \mathcal{T}^{(n-1)}} \Pi_{\bfi} Q_{\bfi} \right) \vee \left( \bigvee_{{\bfi} \in A_n} \Pi_{\bfi} B_{\bfi} \right),$$
where the $\{ B_{\bfi} \}$ are i.i.d.~nonnegative random variables, independent of all other branching vectors in $\mathcal{T}^{(n-1)}$. Then, provided
$$\lim_{x \to \infty} x^\alpha P(B > x) = \gamma \geq 0,$$
the random variable $R(B) = \lim_{n \to \infty} R_n(B)$ is a solution to \eqref{eq:Maximum} (see Theorem~1(ii) in \cite{Biggins_98}). Note that the special endogenous solution $R$ given by \eqref{eq:Endogenous} corresponds to taking the terminal values $\{B_{\bfi}\}$ identically equal to zero, and is known to be the minimal solution in the usual stochastic order sense (see Proposition~5 in \cite{Biggins_98}). Moreover, by Theorem~1(i) in \cite{Biggins_98}, $R$ is finite a.s.~whenever
$$\sup_{x \geq 1} x^\alpha (\log x)^{1+\epsilon} P(Q > x) < \infty$$
for some $\epsilon > 0$.

Besides observing that in applications \cite{Kar_Kel_Suh_94, Olv_Rui_19} it is usually the special endogenous solution that is of interest, it is worth mentioning that it plays an important role in characterizing all the solutions defined through $R(B)$, whose distributions are given by
\begin{equation} \label{eq:GeneralSolution}
P( R(B) \leq x) = E\left[ 1(R \leq x) \exp(-\gamma W(\alpha) x^{-\alpha}) \right],
\end{equation}
where $W(\alpha)$ is the a.s.~limit of the martingale $W_n(\alpha) := \sum_{{\bfi} \in A_n} \Pi_{\bfi}^\alpha$ (see Theorem~1(ii) in \cite{Biggins_98}). The martingale $\{ W_n(\theta): n \geq 1\}$ defined via $W_n(\theta) := \rho_\theta^{-n}  \sum_{{\bfi} \in A_n} \Pi_{\bfi}^\theta$, where $\rho_\theta := E\left[ \sum_{j=1}^N C_j^\theta \right]$, is known as the Biggins-Kyprianou martingale \cite{Biggins_77, Kyprianou_00}, and it plays an important role in much of the weighted branching processes literature.
Moreover, under additional technical conditions, all the solutions to \eqref{eq:Maximum} can be characterized through \eqref{eq:GeneralSolution} (see Theorem~1(iii) in \cite{Biggins_98}).

\begin{example} \upshape
To illustrate the multiplicity of solutions to \eqref{eq:Maximum}, consider the case when $N \equiv 2$, $C_i \equiv \frac{1}{2}$ for $i=1,2$ and $Q \equiv \frac{1}{2}$, whose endogenous solution is given by $R= \bigvee_{{\bfi} \in \mathcal{T}} \Pi_{\bfi} Q_{\bfi} = \frac{1}{2}$. Now note that if $R' = (T \vee 1)/2$ where $T$ has a Frechet distribution with shape/scale parameters $(1,s)$, i.e., $P(T \leq x) = e^{-s/x}$ 
for $x > 0$, then $R'$ is a (non-endogenous) solution since
$$Q  \vee \bigvee_{i=1}^N C_i  R_i' = \frac{1}{2} \vee \bigvee_{i=1}^2 \frac{1}{2} \cdot \frac{(T_i \vee 1)}{2}  = \frac{1}{2} \max\left\{ 1, \frac{T_1 \vee T_2}{2} \right\} \stackrel{\mathcal{D}}{=} \frac{1}{2} ( 1 \vee T) = R'.$$
Furthermore, by setting $B = T/2$ we can identify $R'$ with
$$R(B) = \lim_{n \to \infty} R_n(B) = \lim_{n \to \infty} \frac{1}{2} \vee \left( \bigvee_{{\bfi} \in A_n} \frac{B_{\bfi}}{2^n} \right) =  \lim_{n \to \infty} \frac{1}{2} \max \left\{ 1 , \bigvee_{{\bfi} \in A_n} \frac{T_{\bfi}}{2^n}  \right\} \stackrel{\mathcal{D}}{=} \frac{1}{2} (1 \vee T) = R'.$$
\end{example}

Our analysis of $R$ will rely on a set of assumptions satisfied by the generic branching vector $\bm{\psi} = (N, Q, C_1, C_2, \dots)$. 

\begin{assum} \label{Ass1}
$(N, Q, C_1, C_2, \ldots)$ is nonnegative a.s. with $N \in \mathbb{N}_+ \cup \{\infty\}$, and $P(Q > 0) > 0$.
Furthermore, for some $\alpha>0$,
\begin{enumerate}[(a)]
\item $E \left[  \sum_{i=1}^N C_i ^\alpha \right] = 1$ and $E\left[ \sum_{i=1}^N C_i^\alpha \log C_i \right] \in (0, \infty)$,

\item $E\left[ \sum_{i=1}^N C_i^\beta \right] < 1$ for some $0 < \beta < \alpha$ and $E[Q^\alpha] < \infty$,

\item $P\left( \sum_{i=1}^N C_i^\alpha = 0 \right) = 0$,

\item The probability measure $\eta(dx) = E\left[ \sum_{i=1}^N C_i^\alpha 1(\log C_i \in dx) \right]$ is non-arithmetic,

\item $E\left[\left(\sum_{i=1}^N C_i^\alpha\right) \log^+\left(Q^\alpha \vee \sum_{i=1}^N C_i^\alpha \right)\right] < \infty$.

\end{enumerate}
\end{assum}

Since the approach followed here is different from the one used in the implicit renewal theorem found in \cite{Jel_Olv_15}, our assumptions for establishing the representation of the constant in Theorem~\ref{T.Main} are slightly different. In particular, conditions (c) and (e) are new. Condition (c) will be needed to ensure that our change of measure is well-defined, and condition (e) will guarantee that the positive part of the perturbed branching random walk  has finite mean under said change of measure. On the other hand, the implicit renewal theorem requires the following assumption, which we use only for the positivity of the constant in Theorem~\ref{T.Main}. 

\begin{assum}\label{Ass2}

$E\left[ \left( \sum_{i=1}^N C_i \right)^\alpha \right] < \infty$ if $\alpha > 1$ and $E\left[ \left( \sum_{i=1}^N C_i^{\alpha/(1+\epsilon)} \right)^{1+\epsilon} \right] < \infty$ for some $0 < \epsilon < 1$ if $0 < \alpha \leq 1$.

\end{assum}

Observe that apart from these assumptions, the dependence structure in the vector $\bm{\psi}$ is  completely arbitrary.

\subsection{Change of measure along a path}\label{S.Changeofmeasure}

Although the weighted branching process is more naturally defined in terms of the weights $\Pi_{\bfi}$ along the branches of $\mathcal{T}$, it will be more convenient from this point onwards to focus on the additive version of \eqref{eq:Maximum}. Note that for any path ${\bfi}$ originating at the root of $\mathcal{T}$, we can define a random walk by setting $S_{\bfi} := \log \Pi_{\bfi}$. Moreover, by letting $Y_{\bfi} = \log Q_{\bfi}$, we obtain that
\begin{equation}\label{eq:LogEndogenous}
W := \log R = \bigvee_{{\bfi} \in \mathcal{T}} (S_{\bfi} + Y_{\bfi})
\end{equation}
represents the maximum of a perturbed branching random walk.

Since our goal is to analyze the tail distribution $P(W > t)$ (equivalently, of $P(R > t)$) for all values of $t$, the key idea of our analysis is to apply a change of measure to the perturbed branching random walk under which the event $\{ W > t \}$ for large $t$ is no longer rare. This is exactly the usual approach for studying the maximum of the standard random walk under Cram\'er conditions (i.e., the existence of $\alpha > 0$ such that $E[e^{\alpha X}] = 1$ and $0 < E[X e^{\alpha X}] < \infty$). However, in the branching case the change of measure is not as straightforward as in the non-branching case, where we use the exponential martingale to define it (see Chapter X in \cite{Asm2003}).

Note that under the condition  $E\left[ \sum_{i=1}^N C_i^\beta \right] < 1$ for some $\beta > 0$, the paths in the tree $\mathcal{T}$ have negative drift\footnote{Since Jensen's inequality gives $E\left[ \max_{1 \leq i \leq N} X_i \right] = \beta^{-1} E\left[ \log \left( \bigvee_{i=1}^N C_i^\beta \right) \right] \leq \beta^{-1} \log E\left[ \sum_{i=1}^N C_i^\beta \right] < 0$.}. The change of measure we seek is obtained by making the drift of one path positive.
Starting at the root, we will pick this {\em chosen path} by selecting one of its offspring at random, with a probability proportional to its weight raised to the $\alpha$ power. This procedure will allow us to define a suitable mean one nonnegative martingale to induce a change of measure on the entire tree. As we will show later, the change in the drift will not affect any subtrees whose roots are not part of the chosen path, allowing us to isolate the (small) set of paths responsible for the rare event $\{W > t\}$.

More precisely, let ${\bfJ}_0 = \emptyset$ denote the root of $\mathcal{T}$. We now recursively define the random indices
along the chosen path, $\{ {\bfJ}_k: k \geq 1 \}$, as follows:
$$P\left(  {\bfJ}_k = ({\bfJ}_{k-1}, i) \left| \bm{\psi}_{{\bfJ}_{k-1}} \right. \right) = \frac{C_{({\bfJ}_{k-1}, i)}^\alpha}{D_{{\bfJ}_{k-1}}}, \qquad 1 \leq i \leq N_{{\bfJ}_{k-1}}, \quad k \geq 1,$$
where $D_{\bfi} = \sum_{r=1}^{N_{\bfi}} C_{({\bfi},r)}^\alpha$ for any ${\bfi} \in U$, with generic copy $D$. 
 From now on, we will refer to this chosen path along with its offspring and sibling nodes as the {\em spine}. Note that the sequence of indexes $\{ {\bfJ}_k: k \geq 0 \}$ identifies all the nodes in the spine, with node ${\bfJ}_k$ denoting the one in the $k$th generation of $\mathcal{T}$.

We now use the spine to define a mean one nonnegative martingale for our change of measure. To this end, define
$$
 L_0 = 1, \qquad
 \qquad L_k =
\prod_{r=0}^{k-1} D_{{\bfJ}_r} \,,\qquad k \geq 1 \,,
$$
and note that if we let $\mathcal{F}_k = \sigma( \bm{\psi}_{\bfi}: {\bfi} \in A_s, \, s < k)$ and $\mathcal{G}_k = \sigma( \mathcal{F}_k \cup \sigma({\bfJ}_s: s \leq k) )$ for $k \geq 1$ and $\mathcal{F}_0 = \mathcal{G}_0 = \sigma(\varnothing)$, then
\begin{align*}
E\left[ \left. L_k \right| \mathcal{G}_{k-1} \right] &= L_{k-1} E\left[ \left.  D_{{\bfJ}_{k-1}} \right| \mathcal{G}_{k-1} \right]  
 = L_{k-1}\,.
\end{align*}
Therefore $\{L_k : k \ge 0 \}$ is a nonnegative martingale with mean one, measurable with respect to the filtration $\{ \mathcal{G}_k: k \geq 0\}$. It is worth observing that $\{ L_k: k \geq 0\}$ is different from the Biggins-Kyprianou martingale $\{W_n(\alpha): n \geq 1\}$ (see Section~\ref{S.HighOrderLindley}).
Setting
\begin{equation}\label{eq:Ptilde}
\widetilde P(A) = E[ 1(A) L_{k} ]\,, \qquad \text{ for } A \in \mathcal{G}_k \quad \text{and all } k \geq 0\,,
\end{equation}
we obtain a new probability measure on $\mathcal{G} =\sigma\left(\bigcup_{k\geq 1} \mathcal{G}_k \right)$.
Note in particular that $\bm{\psi}=\bm{\psi}_\emptyset$ satisfies
\[
 \widetilde{P}(\bm{\psi} \in B ) = E  \left[ 1(\bm{\psi} \in B ) L_1  \right]
\]
for Borel sets $B$. It is also important to note that the filtration $\mathcal{G}_k$ is larger than the natural filtration of the martingale $\{L_k: k \geq 0\}$, in particular, it includes the values of the perturbations $\{Q_{\bfi}\}$. Therefore, in order to preserve the absolute continuity of $P$ with respect to $\widetilde P$ 
 for all generic branching vectors, we must ensure that the support of $(N, Q, C_1, C_2, \dots)$ does not change, which precludes the possibility of having $P(L_1 = 0) > 0$. The corresponding condition is given by Assumption~\ref{Ass1}(c).

\begin{remark} \upshape 
It is worth mentioning that both  Goldie's implicit renewal theorem \cite{Goldie_91} and the implicit renewal theorem on trees \cite{Jel_Olv_12a, Jel_Olv_15} allow $P(L_1 = 0) > 0$, which is precluded by Assumption~\ref{Ass1}.  Our current setting is less general because it clearly identifies the most likely path to the rare event $\{ W > t \}$ in cases where it is solely determined by the behavior of the spine. However, the implicit renewal theorems cover cases where the most likely path to the rare event is somewhat different than the one we will describe, which translates into the same exponential decay but with a different constant.
\end{remark}

As mentioned earlier, the change of measure defined above only affects the drift of the random walk and the perturbation along the spine.  Moreover, it preserves the branching property, i.e., the independence between the vectors $\{ \bm{\psi}_{\bfi}: {\bfi} \in \mathcal{T}\}$. The following result formalizes this statement; its proof is given in Section~\ref{S.Proofs}. Throughout the paper we use the convention $\sum_{i=a}^b x_i \equiv 0$ whenever $a > b$.

\begin{lemma}\label{L.Newmeasure}
Suppose Assumption~\ref{Ass1}(a) holds. For any measurable set $B \in \mathbb{N} \times \mathbb{R}^\infty$,
and any ${\bfi} \in A_k$,
\begin{align*}
\widetilde{P}( {\bfJ}_k = {\bfi}) &= \prod_{r=1}^k E[C_{i_r}^\alpha 1(N \geq i_r)],  \\
\widetilde{P}( \bm{\psi}_{\bfi} \in B | {\bfi} \neq {\bfJ}_k) &= P\left( (N, Q, C_1, C_2, \dots) \in B \right) , \\
\widetilde{P}( \bm{\psi}_{\bfi}  \in B | {\bfi} = {\bfJ}_k) &=  E\left[ 1\left( (N, Q, C_1, C_2, \dots) \in B \right) \sum_{j=1}^N C_j^\alpha \right].
\end{align*}
Moreover, under $\widetilde{P}$, the vectors $\{ \bm{\psi}_{\bfi}: {\bfi} \in A_k \}$ are conditionally independent given $\mathcal{G}_{k-1}$ for any $k \geq 1$. 
\end{lemma}

Recall that by taking the logarithm of the weights we can define a perturbed random walk along every path ${ \bf i} \in \mathcal{T}$. The one along the spine will be special, since it is the one being affected by the change of measure, and will be the only one guaranteed to eventually exceed any level $t$. To make this precise, let us define $X_{\bfi} = \log C_{\bfi}$ and note that for any ${\bfi} \in A_k$,
$$S_{\bfi} = \log \Pi_{\bfi} = X_{{\bfi}|1}  + \dots + X_{{\bfi}|k-1} + X_{\bfi},$$
where the $\{ X_{{\bfi}|r} \}_{1 \leq r \leq k}$ are independent of each other, although not necessarily identically distributed. To identify the spine we use the notation
$\hat X_k = X_{{\bfJ}_k} = \log C_{{\bfJ}_k}$, identify the random walk along the chosen path by
\begin{equation}\label{e:Vk}
V_k = \hat X_1 + \dots + \hat X_k, \qquad V_0 = 0,
\end{equation}
and use
$$\xi_k = Y_{{\bfJ}_k} = \log Q_{{\bfJ}_k}, \qquad \xi_0 = Y_\emptyset = \log Q_\emptyset,$$
for its perturbation. The following result establishes that $\{ V_k: k \geq 0\}$ defines a random walk with i.i.d.~increments and positive drift.

\begin{lemma}\label{L.DistOnJ}
Suppose Assumption~\ref{Ass1}(a) holds. For all $k\geq 1$ and $x_1, \ldots x_k, y  \in \mathbb{R}\cup\{\infty\}$, we have
$$\widetilde P\left( \hat X_1 \leq x_1, \dots, \hat X_k \leq x_k, \, \xi_k \leq y  \right) = E\left[ 1(Q \leq e^y) \sum_{i=1}^N C_i^\alpha \right] \prod_{r=1}^k G(x_r) ,$$
where
$$G(x) = \sum_{i =1}^\infty E\left[ 1(C_i \leq e^x, \, N \geq i) C_i^\alpha \right] = E\left[ \sum_{i=1}^N 1(\log C_i \leq x) C_i^\alpha \right].$$
In particular, the $\{ \hat X_i: i \geq 1\}$ are i.i.d.~with common distribution $G$ under $\widetilde{P}$, $\widetilde{E}\left[ |\hat  X_1| \right] < \infty$, and have mean
 $$\mu := \widetilde{E} \left[ \hat X_1\right] = E\left[ \sum_{i=1}^N C_i^\alpha \log C_i  \right]  \in (0, \infty).$$
\end{lemma}

We now explain how to compute the probability $P(W > t)$ using the change of measure described above. We start by defining the hitting time of level $t$ for the perturbed branching random walk defined by $\{ (S_{\bfi}, Y_{\bfi}): {\bfi} \in \mathcal{T}\}$, which we denote $\gamma(t) = \inf\{ {\bfi}  \in \mathcal{T} : S_{\bfi} + Y_{\bfi}  > t\}$, where the infimum is taken according to the $\prec$ order defined in Section~\ref{S.WBP}. We use $\nu(t) = |\gamma(t)|$ to denote the generation in the weighted branching process where the perturbed random walk along a path exceeds level $t$.  Next, define the hitting time of level $t$ along the spine, $\tau(t) = \inf\{ k \geq 0: V_k + \xi_k > t\}$.

Note that $\nu(t)+1$ and $\tau(t)+1$ are stopping times for the weighted branching process with respect to the filtration $\{ \mathcal{G}_k: k \geq 0\}$, and since it is possible for a path different from the spine to hit level $t$ before the spine does, then
$$\nu(t) \leq \tau(t),$$
with equality possible, e.g. if  ${\bfJ}_{\tau(t)} = \gamma(t)$. Moreover, since  $W = \bigvee_{{\bfi} \in \mathcal{T}} (S_{\bfi} + Y_{\bfi})$, it follows that
$$P(W > t) = P(\nu(t) < \infty).$$

The next step is to apply the change of measure and derive an alternative representation for $P(\nu(t) < \infty)$. To this end, observe that on the set
$\{\gamma(t) = {\bfi}\}$, we have $N_{{\bfi} |r-1} \geq i_r$ for all $ r=1,\ldots , k$ and that $\{\gamma(t) = {\bfi},\, {\bfJ}_{k} = {\bfi} \} =
\{\tau(t) = k, \, {\bfJ}_{\tau(t)} = \gamma(t) = {\bfi}\}$.
Also note that for
${\bfi} \in A_k$,
$$
 P\left( \left. {\bfJ}_{k} = {\bfi} \right| \mathcal{F}_k \right) = \frac{ \prod_{r=1}^k  C^\alpha_{{\bfi} |r }  1(N_{{\bfi} |r-1} \geq i_r) }{ \prod_{r=0}^{k-1} D_{{\bfi}|r}}\,.
$$
Therefore, since $P\left( \left. {\bfJ}_{k} = {\bfi} \right| \mathcal{F}_k \right)  = P\left( \left. {\bfJ}_{k} = {\bfi} \right| \mathcal{F}_{k+1} \right)$, and since $\prod_{r=0}^{k-1} D_{{\bfi}|r} > 0$ a.s. for all ${\bfi} \in A_k$ and all $k$, we have
\begin{align} \nonumber
e^{\alpha t} P(W > t) &= E\left[ \sum_{k=0}^\infty  \sum_{{\bfi} \in A_k} e^{\alpha t} 1(\gamma(t) = {\bfi}) \right] \\\nonumber
&= \sum_{k=0}^\infty E\left[ \sum_{{\bfi} \in A_k} 1(\gamma(t) = {\bfi}) e^{\alpha t} \frac{\Pi_{\bfi}^\alpha}{\Pi_{\bfi}^\alpha} \cdot
\frac{ \prod_{r=0}^{k-1} D_{{\bfi}|r} }{ \prod_{r=0}^{k-1} D_{{\bfi}|r}} \right] \\\nonumber
&= \sum_{k=0}^\infty E\left[ \sum_{{\bfi} \in A_k} 1(\gamma(t) = {\bfi}) e^{-\alpha (S_{\bfi} - t) }   \cdot \prod_{r=0}^{k-1} D_{{\bfi}|r} \cdot  P\left( \left. {\bfJ}_{k} = {\bfi}  \right|  \mathcal{F}_{k+1} \right)  \right] \\\nonumber 
&= \sum_{k=0}^\infty E\left[ E\left[ \left. \sum_{{\bfi} \in A_k} 1(\gamma(t) = {\bfi})  e^{-\alpha (S_{\bfi} - t) }   L_k  1( {\bfJ}_{k} = {\bfi})   \right|  \mathcal{F}_{k+1} \right]  \right] \\
&= \sum_{k=0}^\infty E\left[  1(\gamma(t) = {\bfJ}_{\tau(t)}, \, \tau(t) = k )e^{-\alpha (V_{\tau(t)} - t) }   L_k       \right] .\label{eq:Eoverk}
\end{align}
Now note that although $\tau(t)$ and $|\gamma(t)|$ are not stopping times with respect to $\{\mathcal{G}_k: k \geq 0\}$, $\tau(t) + 1$ and $|\gamma(t)| + 1$ are. Hence, multiplying and dividing by $D_{{\bfJ}_{k}}$ we obtain
\begin{align*}
e^{\alpha t} P(W > t) &= \sum_{k=0}^\infty E\left[  1(\gamma(t) = {\bfJ}_{\tau(t)}, \, \tau(t)+1 = k+1) e^{-\alpha (V_{\tau(t)} - t) } D_{{\bfJ}_{k}}^{-1}   L_{k+1}       \right] \\
&=  \widetilde{E}\left[  1(\gamma(t) = {\bfJ}_{\tau(t)} , \tau(t) < \infty) e^{-\alpha (V_{\tau(t)} - t) } D_{{\bfJ}_{\tau(t)}}^{-1}         \right].
\end{align*}

We will show that since $\{V_k + \xi_k: k \geq 0\}$ is a perturbed random walk with positive drift under $\widetilde{P}$, then, under Assumption~\ref{Ass1}(a)-(c),  $\widetilde P(\tau(t) < \infty) = 1$ for all $t$ (see Lemma~\ref{L.nuFinite}). Hence, we obtain
\begin{equation}  \label{eq:ISEstimator}
	e^{\alpha t}P(W > t) = \widetilde E \left[ 1({\bfJ}_{\tau(t)} = \gamma(t) ) e^{-\alpha (V_{\tau(t)} - t) }  D_{{\bfJ}_{\tau(t)}}^{-1} \right]\,, 
\end{equation}
where the expectation on the right-hand side no longer vanishes as $t \to \infty$. Note that the right-hand-side of \eqref{eq:ISEstimator} is an explicit function of the first $\tau(t)$ generations of a weighted branching process with a distinguished spine, which can be directly estimated using standard Monte Carlo methods, as discussed in Section~\ref{S.Simulation}.

\begin{remark} \label{R.MainRemarks} \upshape 
\begin{enumerate}[(a)]
\item Note that if $Q$ is independent of $(N, C_1, C_2, \dots)$, then we can use the filtration $\mathcal{F}_0' = \sigma(Q_\emptyset)$, $\mathcal{F}_k' = \sigma \left( \psi_{\bfi}: {\bfi} \in A_s, s < k; Q_{\bfj}: {\bfj} \in A_k \right)$ and its corresponding $\mathcal{G}_0' = \mathcal{F}_0'$,  $\mathcal{G}_k' = \sigma\left( \mathcal{F}_k' \cup \sigma({\bfJ}_s: s \leq k) \right)$, with respect to which both $\tau(t)$ and $\nu(t)$ are stopping times, and obtain the simpler expression

$$e^{\alpha t} P(W > t) = \widetilde{E}\left[ 1({\bfJ}_{\tau(t)} = \gamma(t)) e^{-\alpha(V_{\tau(t)} - t)} \right].$$

\item In the non-branching case $(N \equiv 1)$, equation \eqref{eq:ISEstimator} reduces to
\begin{equation} \label{eq:NonBranching}
P(W > t) = \widetilde{E} \left[ e^{-\alpha V_{\tau(t)+1} }  \right],
\end{equation}
which we point out is different from equation (3.4) in \cite{Araman}, since their expression has $V_{\tau(t)}$ instead of $V_{\tau(t)+1}$. As we explained earlier, $\tau(t)$ is not a stopping time with respect to the natural filtration $\mathcal{H}_n = \sigma( \hat X_i: 1 \leq i \leq n)$ of the martingale $L_n$, so the change of measure argument in \cite{Araman} needs to be modified (see Theorem~3.2 in Chapter XIII of \cite{Asm2003}). Once we consider the augmented filtration $\mathcal{F}_k$ (which is equal to $\mathcal{G}_k$ in the non-branching case) and apply the change of measure up to the stopping time $\tau(t)+1$, we obtain the expression given by \eqref{eq:NonBranching}.  

\item The case where the $Q$ is bounded is also special in the sense of the theory needed for its analysis. In particular, the exponential asymptotics of $P(W > t)$ can be easily obtained without using the augmented filtration nor any implicit renewal theory. To illustrate this we include in Section~\ref{S.Proofs} (see Theorem~\ref{T.BoundedQNonBranching}) a very short proof of Theorem~1 in \cite{Araman}, for the non-branching case. Since the focus of the current paper is to obtain a more explicit representation for the constant $H$ obtained through the implicit renewal theorem on trees (Theorem~3.4 in \cite{Jel_Olv_15}), we do not pursue the bounded $Q$ case separately in the branching setting. 

\item Moreover in the case of a.s. bounded $Q$, say $P(Q \leq q) = 1$, we can obtain a {\em Cram\'er-Lundberg} type of inequality for $P(W > t)$ by defining $\gamma^*(t) = \inf\{ {\bfi} \in \mathcal{T}: S_{\bfi} > t\}$, $\nu^*(t) = |\gamma^*(t)|$, and $\tau^*(t) = \inf\{ n \geq 1: V_n > t\}$, and noting that both $\nu^*(t)$ and $\tau^*(t)$ are stopping times with respect to the filtration $\mathcal{H}_n= \sigma( \{(N_{\bf i}, C_{({\bf i}, 1)}, C_{({\bf i}, 2)}, \dots): {\bf i} \in A_s, \, s < k\}, \{{\bf J}_s: s \leq k\})$.
The same change of measure arguments used above yield for any $t > c := \log q$:
\begin{align*}
P(W > t) &= P( \nu^*(t-c) \leq \nu(t) < \infty) \\
&\leq P(\nu^*(t-c) < \infty) \\
&= \widetilde{E}\left[ 1({\bfJ}_{\tau^*(t-c)} = \gamma^*(t-c)) e^{-\alpha V_{\tau^*(t-c)}} \right] \\
&\leq \widetilde{E}\left[ e^{-\alpha (V_{\tau^*(t-c)} - t + c)} \right] e^{-\alpha(t-c)} \\
&\leq q^\alpha e^{-\alpha t}.
\end{align*}
This inequality for all $t \geq c$ holds under Assumption~\ref{Ass1}(a), and cannot be obtained using only the implicit renewal theorem for trees in \cite{Jel_Olv_15}.

\end{enumerate}
\end{remark}

\subsection{The Markov renewal theorem}

As pointed out, the new expression provided by \eqref{eq:ISEstimator} can easily be estimated via simulation. however, it can also be directly analyzed to obtain an alternative representation for the constant $H$ in $P(W > t) \sim H e^{-\alpha t}$, $t \to \infty$. The idea behind this analysis is the use of renewal theory on the expectation
$$\widetilde E \left[ 1({\bfJ}_{\tau(t)} = \gamma(t) ) e^{-\alpha (V_{\tau(t)} - t) }  D_{{\bfJ}_{\tau(t)}}^{-1}  \right].$$ Note that although the exponential term inside the expectation depends only on the random walk $\{ V_k: k \geq 0\}$ and its hitting time of level $t$, the event $\{ {\bfJ}_{\tau(t)} = \gamma(t)\}$ depends on the history of the tree  $\mathcal{T}$ up to generation $\nu(t)$. Hence, any renewal argument would need to include the latter, which complicates matters since its exponential growth (whenever $E[N] > 1$) implies it does not naturally renew at any point. However, intuitively, only the paths that branch out from the spine close to the time when the spine is likely to reach level $t$ are likely to reach level $t$ at all. This means that it should suffice to focus only on these paths, say a subtree of height $m$ rooted at the spine that moves along the random walk $\{ V_k:  k \geq 0\}$; see Figure~\ref{F.Spine}. Since the sequence of such height-$m$ subtrees forms a Harris chain, the key to our main theorem is the use of the Markov renewal theorem in \cite{Alsmeyer_94}.

\begin{figure}[t]
\centering
\begin{subfigure}[t]{0.45\textwidth}
	\centering
	\begin{picture}(250,180)(0,0)
		\put(0,0){\includegraphics[scale = 0.5, bb = 10 385 410 785, clip]{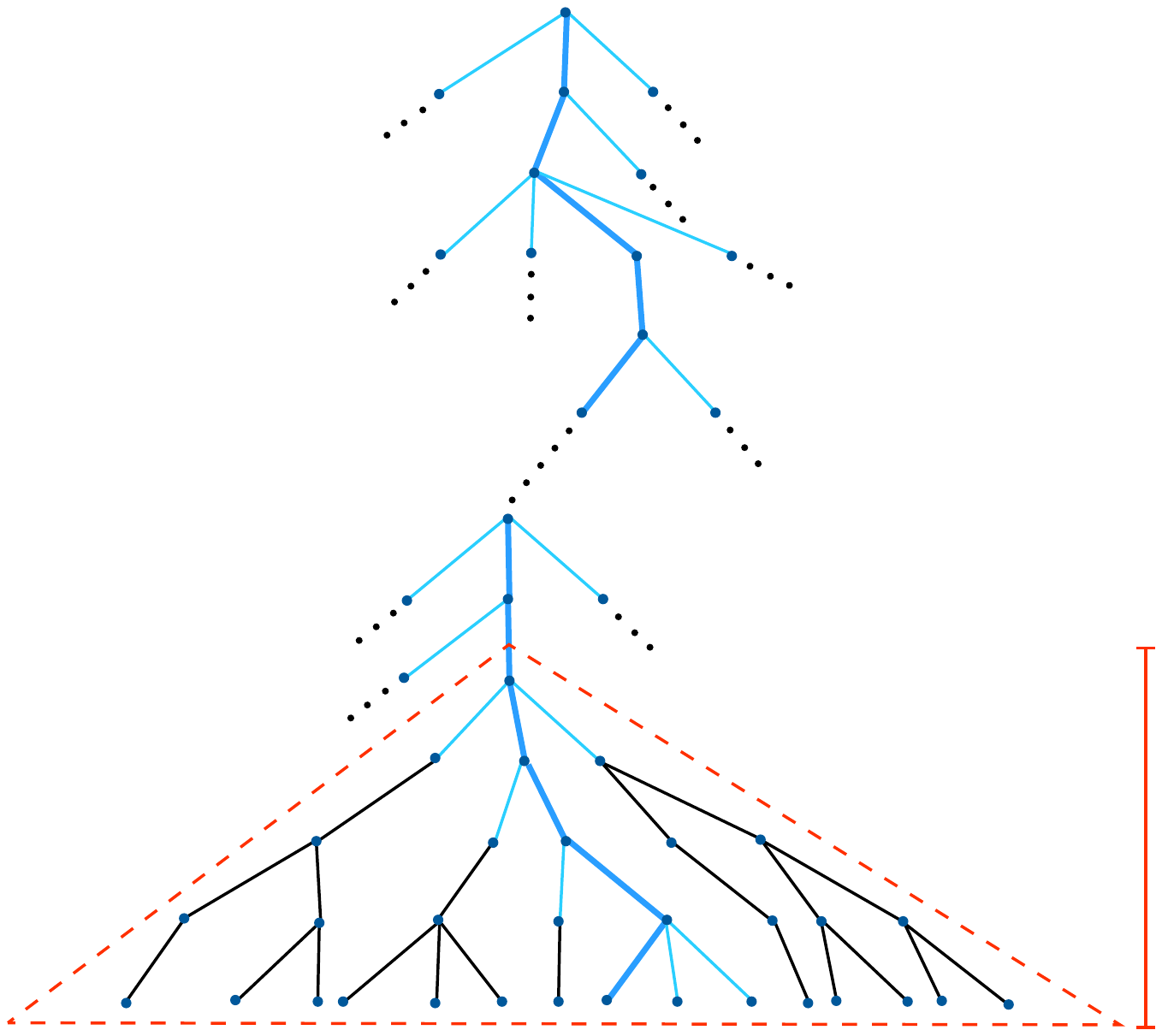}}  
		\put(102,177){\footnotesize $\emptyset$} 
		\put(205,60){\rotatebox{-90}{\footnotesize $m$ generations}} 
	\end{picture}
	\caption{The spine of $\mathcal{T}$.} \label{F.Spine}
\end{subfigure}
\hspace{10mm}
\begin{subfigure}[t]{0.45\textwidth}
	\centering
	\begin{picture}(250,180)(0,0)
		\put(0,0){\includegraphics[scale = 0.5, bb = 10 385 410 785, clip]{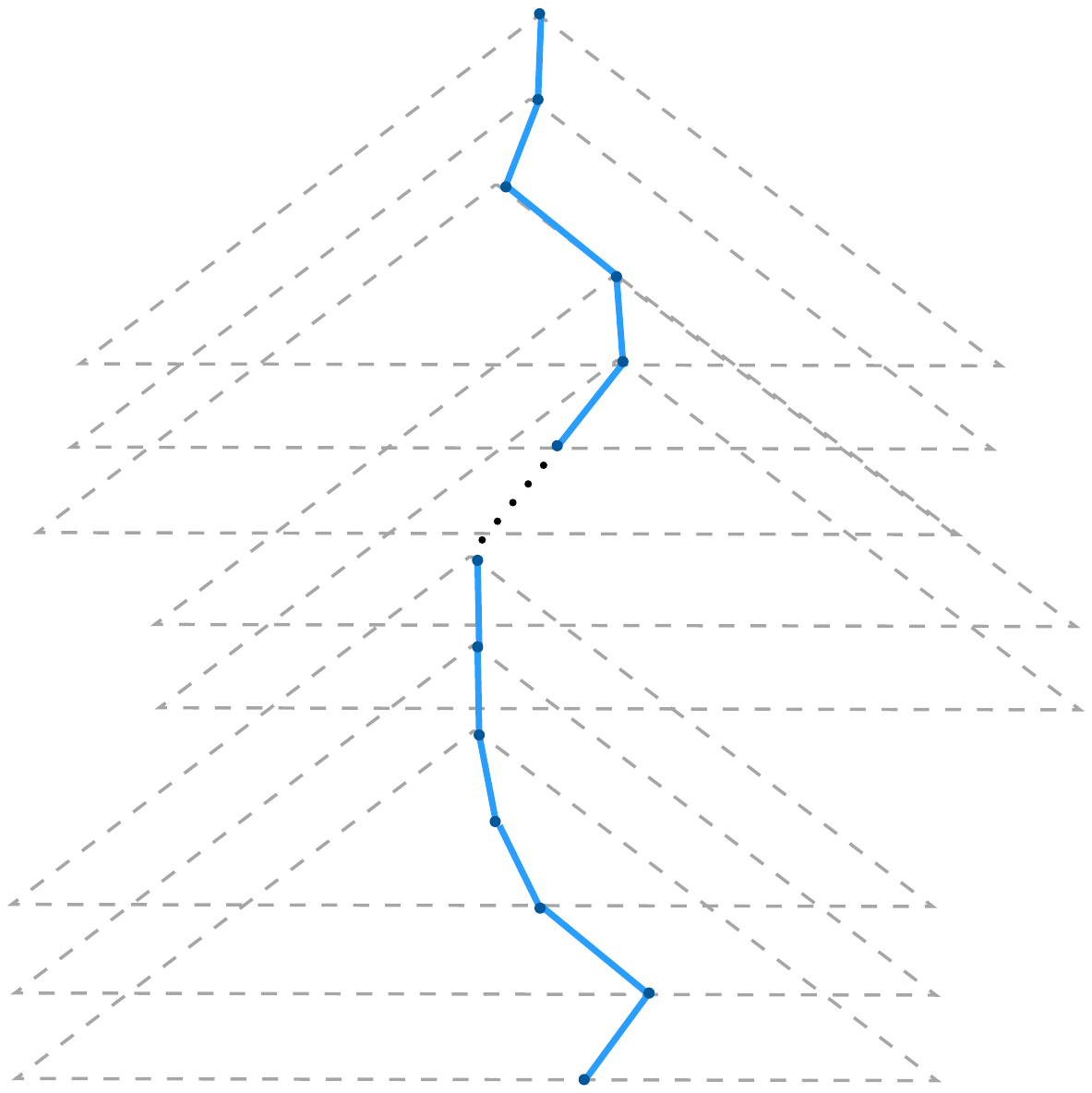}}
		\put(102,177){\footnotesize $\emptyset$}
	\end{picture}
	\caption{The subtrees of height $m$ rooted in the spine, $\mathcal{T}^{(m)}_k$, $k \geq 0$.} \label{F.SpineHarrisChain}
\end{subfigure}
\caption{The spine of $\mathcal{T}$ and the Markov chain consisting of subtrees.}
\end{figure}

To formalize this idea, we define the subtrees of height $m$ rooted at node ${\bfJ}_k$ (the $k$th node along the spine) according to:
\begin{equation} \label{eq:T_nDef}
\mathcal{T}^{(m)}_k = \bigcup_{n=0}^{m-1} A_{n,{\bfJ}_k}, \qquad k \geq 0,
\end{equation}
where $A_{n,{\bfi}} = \{ ({\bfi}, {\bfj}) \in \mathcal{T}: |{\bfj}| = n\}$ is the $n$th generation of the subtree rooted at node ${\bfi}$ (See Figure~\ref{F.SpineHarrisChain}). Focusing on these subtrees rooted at the spine allows us to analyze the expectation $\widetilde E \left[ 1({\bfJ}_{\tau(t)} = \gamma(t) ) e^{-\alpha (V_{\tau(t)} - t) } D_{{\bfJ}_{\tau(t)}}^{-1}  \right]$ using the Markov renewal theorem  in \cite{Alsmeyer_94}.
 Note that even in the non-branching case $(N \equiv 1)$, the perturbations that the $Q$'s represent make it difficult to identify clear regeneration epochs for the process $\{V_k+\xi_k: k \geq 0\}$, which is a problem that is solved by looking not only at the current value of $V_n + \xi_n$, but also at its $m$-step history. We give more details on this idea and the intuition behind it in Section~\ref{S.Proofs}.

\section{Main result}\label{S.MainResult}

We are now ready to present our main theoretical result. Recall $\mu = \widetilde{E}[\hat X_1] = E\left[ \sum_{i=1}^N C_i^\alpha \log C_i \right]$.

\begin{theo}\label{T.Main}
If $(N, Q, C_1, C_2, \ldots)$ satisfies Assumption \ref{Ass1} for some $\alpha > 0$ and $W = \log R$, where $R$ is the endogenous solution to \eqref{eq:Maximum} given by \eqref{eq:Endogenous}, then
\[
	P(W > t) \sim H e^{-\alpha t} \quad \mbox{as} \quad t \to \infty,
\]
where
\[
	H = \lim_{m\to\infty} \frac{\widetilde E \left[ \left( e^{\alpha  \xi_{m}} - e^{\alpha \left( \bigvee_{{\bfi} \prec {\bfJ}_{m}} (S_{\bfi} + Y_{\bfi}) - V_m \right)} \right)^+ D_{{\bfJ}_m}^{-1}  \right] }{\alpha \mu}.
\]
If furthermore Assumption \ref{Ass2} holds, then $H > 0$. 
\end{theo}

\begin{remark}\label{R.SameConstant} \upshape 
It it interesting to compare the expression for $H$ in the theorem with its counterpart obtained through the use of the implicit renewal theorem on trees (Theorem~3.4 in \cite{Jel_Olv_15}), which written in terms of $W$ under our current assumptions\footnote{Theorem~3.4 in \cite{Jel_Olv_15} allows $Q$, and therefore $R$, to take negative values.} becomes
\begin{equation}\label{eq:EquivH}
	\frac{E\left[ e^{\alpha Y} \vee \bigvee_{i=1}^N e^{\alpha (X_i + W_i)} - \sum_{i=1}^N e^{\alpha (X_i + W_i)} \right]}{\alpha \mu},
\end{equation}
where the $\{W_i\}$ are i.i.d.~copies of $W$ independent of the vector $(N, Y, X_1, X_2, \dots)$. As we can see, the two representations are significantly different, despite the fact that they are necessarily equal to each other. However, the representation given by Theorem~\ref{T.Main} applies only to our setting where the rare event is determined by the spine, which under $\widetilde{P}$ behaves very differently than all other paths in the tree, while the constant obtained through the implicit renewal theorem on trees also works for the case where $P(L_1 = 0) > 0$.
\end{remark}

\section{An importance sampling algorithm} \label{S.Simulation}

The same exponential change of measure used to establish \eqref{eq:CramLundAsymptotic} in the non-branching case is well-known to provide an unbiased and strongly efficient estimator for the rare event probability $P(W>t)$ when $t$ is large. Throughout this section we assume that $N < \infty$ a.s.

To relate our estimator to the one used in the non-branching case, suppose first that the goal is  to estimate the tail distribution of the all-time maximum of the random walk $S_n = X_1 + \dots + X_n$ when $E[X_1] < 0$.  For large values of $t$, estimating $P(W  > t) = P(\sup_{n} S_n > t)$ using the naive estimator $1(W > t)$ would require prohibitively large sample sizes, since its relative error grows unboundedly, i.e.,
\[
	\frac{\var(1(W>t))}{P(W>t)^2} = \frac{P(W>t)P(W\le t)}{P(W>t)^2}
	\to \infty \quad\mbox{as}\quad t \to\infty,
\]
However, whenever there exists $\alpha > 0$ such that $E[e^{\alpha X_1}] = 1$ and $E[X_1 e^{\alpha X_1} ] \in (0, \infty)$, Siegmund's algorithm \cite{Siegmund} takes advantage of the representation
$$P(W > t) = \widetilde{E}\left[ e^{-\alpha S_{\tau(t)}} 1(\tau(t) < \infty)\right],$$
where the expectation is computed under the change of measure $\widetilde{P}(A) =E\left[ 1(A) e^{\alpha S_n}\right]$ for any set $A$ measurable with respect to $\sigma(X_1, \dots, X_n)$. Since under $\widetilde{P}$ the random walk has positive drift, $\widetilde{P}(\tau(t) < \infty) = 1$ and we obtain the estimator:
\[
	Z(t) = e^{-\alpha S_{\tau(t)}}.
\]
This estimator is known to be strongly efficient, in the sense that it has bounded relative error, i.e,
\[
	\limsup_{t \to \infty} \frac{\widetilde{\var}(Z(t))}{P(W>t)^2}  < \infty,
\]
where $\widetilde{\var}(\cdot)$ denotes the variance under $\widetilde{P}$.  Furthermore, since it can be shown that $t/\tau(t) \to \mu = E\left[X_1 e^{\alpha X_1} \right]$ $\widetilde{P}$-a.s., then computing $Z(t)$ requires that we simulate around $t/\mu$ steps of the random walk.
 For further details we refer to Chapter VI in \cite{asmussen2007stochastic}.

Our proposed simulation approach for the branching case follows the same ideas described above. However, the issues we encounter while using a naive Monte Carlo approach are considerably worse, since simulating $k$ generations of a tree requires, in general, an exponential in $k$ number of random variables. 
Observe that in similar situations, the population dynamics algorithm \cite{Aldo_Band_05, Mezard_Montanari_2009, Olv_19} has been used to construct
 dependent samples which still yield strongly consistent estimators. However,
our problem here is that  we are interested in estimating the probability $P(W > t)$ for both moderate and large values of $t$, and in the latter case the size of such samples would again have to be prohibitively large in order to obtain enough observations larger than $t$.

Alternatively, we could try to estimate the expectation in the asymptotic expression
$$P(W > t) \sim \frac{E\left[ e^{\alpha Y} \vee \bigvee_{i=1}^N e^{\alpha(X_i + W_i)} - \sum_{i=1}^N e^{\alpha (X_i + W_i)} \right]}{\alpha \mu} \cdot e^{-\alpha t}, \qquad t \to \infty,$$
provided by Theorem~3.4 in \cite{Jel_Olv_15}, since the population dynamics algorithm could be used to efficiently and accurately estimate it. However, we would still have a bias due to the limit in $t$ that cannot be explicitly computed, despite the availability of convergence rates in the implicit renewal theorem \cite{Jel_Olv_13}.  Instead, our proposed estimator follows the idea behind Siegmund's algorithm and is based on the representation
\[
	P(W > t) = \widetilde{E}\left[ 1({\bfJ}_{\tau(t)} = \gamma(t)) e^{-\alpha V_{\tau(t)}} D_{{\bfJ}_{\tau(t)}}^{-1}  \right], 
\]
derived in Section~\ref{S.Changeofmeasure}.  Note that under Assumption~\ref{Ass1}(a)-(c) we have $\widetilde{P}(\nu(t) \leq \tau(t) < \infty) = 1$, which suggests the estimator
\begin{equation} \label{eq:OurEstimator}
Z(t) = 1({\bfJ}_{\tau(t)} = \gamma(t)) e^{-\alpha V_{\tau(t)}} D_{{\bfJ}_{\tau(t)}}^{-1}  ,
\end{equation}
where the underlying tree $\mathcal{T}$ is simulated under the measure $\widetilde{P}$ up to the stopping time $\tau(t)+1$.

\begin{remark}\label{R.ReduceToSiegmund} \upshape
By the discussion in Remark~\ref{R.MainRemarks}(a), when $Q$ is independent of $(N, C_1, C_2, \ldots)$, the estimator 
\begin{equation}\label{eq:IndepQEstimator}
	Z(t) =  1({\bfJ}_{\tau(t)} = \gamma(t)) e^{-\alpha V_{\tau(t)}}
\end{equation}
is also unbiased for $P(W > t)$. This is the prefered estimator in this case since the $D^{-1}_{{\bfJ}_{\tau(t)}}$ in \eqref{eq:OurEstimator} is an unnecessary, independent source of variability. 
In the case that both $N \equiv 1$ and $Q$ is independent of $C_1$, \eqref{eq:IndepQEstimator} reduces to the estimator in Siegmund's algorithm. 
\end{remark}

As with the non-branching case, we expect the spine to reach level $t$ in about $t/\mu$ steps, or equivalently, $t/\mu$ generations of $\mathcal{T}$. The precise result is stated below; note that its proof is not a straightforward consequence of the  strong law of large numbers due to the presence of the perturbations.

\begin{lemma}\label{L.Tovermu}
Under Assumption~\ref{Ass1}, $\tau(t) \to \infty$ $\widetilde P$-a.s. as $t\to\infty$. In particular, $\tau(t) \sim t/\mu$ as $t \to \infty$ $\widetilde{P}$-a.s.
\end{lemma}

Just as the estimator in Siegmund's algorithm, our proposed estimator is strongly efficient, although under a strengthened moment condition due to the perturbations.

\begin{lemma}\label{L.BddRelError}
Suppose Assumptions~\ref{Ass1} and \ref{Ass2} hold, so that $H > 0$. If $E\left[Q^{2\alpha}D^{-1}\right] < \infty$, 
then $Z(t)$ as defined by \eqref{eq:OurEstimator} has bounded relative error. If $Q$ is independent of $(N, C_1, C_2, \ldots)$ and $E\left[Q^{2\alpha}\right] < \infty$, then $Z(t)$ in \eqref{eq:IndepQEstimator} has bounded relative error. 
\end{lemma}

In Table~\ref{T.Algorithm} we present an algorithm for simulating one copy of $Z(t)$ for fixed $t > 0$. At the start, we assume we have computed the value of $\alpha$ such that $E\left[ \sum_{i=1}^N C_i^\alpha \right] = 1$ as well as the corresponding tilted distribution for the nodes along the spine under $\widetilde{P}$, and that we are capable of simulating $(N, Q, C_1, C_2, \dots)$ both under $P$ and under the tilted measure. 
To distinguish the two distributions, let $\bm{\tilde \psi} = (\tilde N, \tilde Q, \tilde C_1, \tilde C_2, \dots)$ denote a vector having the tilted distribution:
\[
	 \widetilde{P}( \bm{\tilde \psi} \in B) = \widetilde P\left(\bm{\psi}_{\bfi} \in B | {\bfi} = {\bfJ}_k \right) = E\left[ 1(\bm{\psi} \in B) \sum_{i=1}^N C_j^\alpha \right], \qquad B \subseteq \mathbb{N} \times \mathbb{R}^\infty,
\]
and let $\bm{\psi} = (N, Q, C_1, C_2, \dots)$ denote a vector having the original distribution under $P$; the simulation of the tree $\mathcal{T}$ will always be done under $\widetilde P$.

\begin{table}
\caption{Importance Sampling Algorithm}
\begin{center}
\begin{tabular}{ll} \hline
1: & {\bf Input:} $t > 0$\\
2: & {\bf Output:} A single copy of $Z = 1({\bfJ}_{\tau(t)} = \gamma(t))e^{-\alpha V_{\tau(t)}}D^{-1}_{{\bfJ}_{\tau(t)}}$\\
3: & Generate  $\left( N, Q, C_{1}, \ldots, C_{N} \right) \stackrel{\mathcal{D}}{=} \bm{\tilde \psi}$ \\
4: & Choose $j \in \{1, \ldots, N \}$ w.p. $C_j^\alpha/\sum_{i=1}^{N} C_i^\alpha$ and set ${\bfJ}_1 \leftarrow j $ \\
5: & Set $Y = \log Q$ and $S_{j} \leftarrow \log C_{j}$ for $j = 1, \ldots, N$ \\
6: & Initialize $S_{\emptyset} \leftarrow 0$, $Y_\emptyset \leftarrow Y$,  ${\bfi} \leftarrow \emptyset$, ${\bfJ}_0 \leftarrow \emptyset$ \\
7: & {\bf while}
 $S_{\bfi} + Y_{\bfi} \leq t$ {\bf do} \\
8: & \hspace*{0.2in} Update ${\bfi} \leftarrow \min\{{\bfj} : {\bfi} \prec {\bfj}\}$ \\
9: & \hspace*{0.2in} {\bf if} ${\bfi} = {\bfJ}_{|{\bfi}|}$ {\bf then}\\
10: & \hspace*{0.4in} Generate $\left( N_{\bfi}, Q_{\bfi}, C_{({\bfi}, 1)}, \ldots, C_{({\bfi}, N_{\bfi})} \right) \stackrel{\mathcal{D}}{=} \bm{\tilde \psi}$ \\
11: & \hspace*{0.4in} Choose  $j \in \{1, \ldots, N_{\bfi}\}$ w.p.  $C_{({\bfi},j)}^\alpha / \sum_{i=1}^{N_{\bfi}} C_{({\bfi},i)}^\alpha$, and set ${\bfJ}_{|{\bfi}|+1} = ({\bfJ}_{|{\bfi}|}, j)$\\
12: & \hspace*{0.2in} {\bf else}\\
13: & \hspace*{0.4in} Generate $\left( N_{\bfi}, Q_{\bfi}, C_{({\bfi}, 1)}, \ldots, C_{({\bfi}, N_{\bfi})} \right) \stackrel{\mathcal{D}}{=} \bm{\psi}$\\
14: & \hspace*{0.2in} {\bf end if}\\
15: &  \hspace*{0.2in} Set $Y_{\bfi} \leftarrow \log Q_{\bfi}$ and $S_{({\bfi}, j)} \leftarrow S_{\bfi} + \log C_{({\bfi}, j)}$ for $j = 1, \ldots, N_{\bfi}$\\
16: & {\bf end while}\\
17: & {\bf if} ${\bfi} = {\bfJ}_{|{\bfi}|}$ {\bf then} \\
18: & \hspace*{0.2in} Set $Z \leftarrow e^{-\alpha S_{\bfi}} / \sum_{i=1}^{N_{\bfi}} C^\alpha_{({\bfi}, i)}$\\
19: & {\bf else} \\
20: & \hspace*{0.2in} Set $Z \leftarrow 0$\\
21: & {\bf end if} \\
22: & Output $Z$ \\ \hline
\end{tabular}
\end{center}
\label{T.Algorithm}
\end{table}

\subsection{Examples}

We now illustrate the use of our proposed simulation algorithm by providing some examples for which both the random vectors $\boldsymbol{\psi}$ and $\boldsymbol{\tilde \psi}$ can be easily simulated.  The particular form of the change of measure poses a simulation challenge since the tilt introduces dependence between $N$ and the $\{ C_i\}$ even if none exists under $P$.  We start with three generic approaches for simulating $\boldsymbol{\tilde \psi}$ and then provide more concrete examples.

\begin{example}[Acceptance-rejection for bounded $C$'s, part I] \label{Ex:BoundedCs} \upshape When the $C_i$ are a.s. bounded, an acceptance-rejection algorithm based on the original distribution of $\boldsymbol{\psi}$ under $P$ can be employed to generate a sample of $\boldsymbol{\tilde \psi}$. Suppose that $C_i \le b_i$ a.s.~for each $i$ and note that
\[
	\widetilde P(\tilde N = n) = E\left[ 1(N=n) \sum_{i=1}^N C_i^\alpha \right] = P(N=n) \sum_{i=1}^n E\left[\left. C_i^\alpha \right| N = n \right], 
\]
so that 
\begin{align*}
	&\widetilde P(\tilde Q \in dy,  \tilde N=n, \tilde C_1\in dx_1, \dots, \tilde C_n \in dx_n) \\
	&= P(N=n) E\left[ \left.1(Q\in dy, C_1\in dx_1, \dots, C_n \in dx_n) \sum_{i=1}^n C_i^\alpha \right| N=n \right] \\
	&= P(N=n)\left(\sum_{i=1}^n E[C_i^\alpha|N=n]\right) \cdot \frac{ E\left[ \left.1(Q\in dy, C_1\in dx_1, \dots, C_n \in dx_n) \sum_{i=1}^n C_i^\alpha \right| N=n \right] }{\sum_{i=1}^n E[C_i^\alpha|N=n]} \\
	&= \widetilde P( \tilde N = n)  \cdot \frac{ E\left[ \left.1(Q\in dy, C_1\in dx_1, \dots, C_n \in dx_n) \sum_{i=1}^n C_i^\alpha \right| N=n \right] }{\sum_{i=1}^n E[C_i^\alpha|N=n]}. 
\end{align*}
Thus, the conditional density of $(\tilde Q, \tilde C_1, \ldots, \tilde C_n)$ given $\tilde N=n$ can be dominated as follows\textcolor{applegreen}{:}
\begin{align*}
	 f_{\tilde Q, \tilde C_1,\ldots, \tilde C_n| \tilde N=n}(y, x_1, \ldots, x_n) &= \frac{\sum_{i=1}^n x_i^\alpha}{\sum_{i=1}^n E[C_i^\alpha|N=n]} f_{Q, C_1,\ldots,C_n|N=n}(y, x_1, \ldots, x_n) \\
	&\le \frac{\sum_{i=1}^n b_i^\alpha}{\sum_{i=1}^n E[C_i^\alpha|N=n]} f_{Q, C_1,\ldots,C_n|N=n}(y, x_1, \ldots, x_n),
\end{align*}
where $f_{Q, C_1,\ldots,C_n|N=n}$ denotes the conditional density of $(Q, C_1, \ldots, C_n)$ given $N=n$ with respect to $P$. Hence, after obtaining $\tilde N = n$ by simulation, an observation of $(\tilde Q, \tilde C_1, \ldots, \tilde C_n)$ can be obtained by using an acceptance-rejection procedure where we simulate $U$ from a $\mbox{Uniform}[0,1]$ distribution and $(Q, C_1, \dots, C_n)$ according to $f_{Q, C_1, \dots, C_n|N= n}$, independent of each other, and then set $(\tilde Q, \tilde C_1, \dots, \tilde C_n) = (Q, C_1, \dots, C_n)$ if
\[
	U \leq \frac{\sum_{i=1}^n C_i^\alpha}{\sum_{i=1}^n b_i^\alpha}.
\]
The acceptance probability given $\tilde N = n$ is $\left( \sum_{i=1}^n b_i^\alpha \right)^{-1}$.

\end{example}

\begin{example}[Acceptance-rejection for bounded $C$'s, part II]  \upshape Suppose that rather than having each of the $C_i$ be bounded individually, we have that $D = \sum_{i=1}^N C_i^\alpha \leq b$ a.s. Now let $Z\sim$ Pareto($a,1$) be independent of $(N, Q, C_1, \dots, C_N)$, and let $f_{Q, C_1, \dots, C_n|N = n}$ be the conditional density of $(Q, C_1, \dots, C_n)$ given $N = n$ with respect to $P$, as in Example~\ref{Ex:BoundedCs}. Now note that the conditional density of $(\tilde Q, \tilde C_1, \ldots, \tilde C_n)$ given $\tilde N=n$ satisfies
\begin{align*}
f_{\tilde Q, \tilde C_1,\ldots, \tilde C_n| \tilde N=n}(y, x_1, \ldots, x_n) &=  \frac{E\left[ \left. 1(Q\in dy, C_1 \in dx_1, \dots, C_n \in dx_n) D \right| N = n \right]}{E[D|N=n]} \\
&= \frac{E\left[ \left. 1(Q\in dy, C_1 \in dx_1, \dots, C_n \in dx_n) b^{-1} D \right| N = n \right]}{E\left[ \left.b^{-1}D\right| N=n \right] } \\
&= \frac{E\left[ \left. 1(Q\in dy, C_1 \in dx_1, \dots, C_n \in dx_n) 1(Z^a > b/D ) \right| N = n \right]}{P(Z^a > b/D | N=n) },
\end{align*}
where we have used the observation that $P(Z^a > b/D  | D) = D/b$ and $P(Z^a > b/D|N=n) = E[D/b|N=n]$. Therefore, after simulating $\tilde N = n$, we can obtain $(\tilde Q, \tilde C_1, \dots, \tilde C_n)$ by generating $Z\sim$ Pareto($a,1$) and $(Q, C_1, \dots, C_n)$ according to $f_{Q, C_1, \dots, C_n|N=n}$, independent of each other,  and then setting $(\tilde Q, \tilde C_1, \dots, \tilde C_n) = (Q, C_1, \dots, C_n)$ if $Z > (b/D)^{1/a}$. The acceptance probability given $\tilde N = n$ is $P(Z > (b/D)^{1/a}) = b^{-1}$.
\end{example}

\begin{example}\label{Ex:Mixture}[A mixture representation] \label{E.InversionTransf} \upshape  The change of measure induces a mixture density in the following way. 
If $\alpha > 0$ is such that $E\left[\sum_{i=1}^N C_i^\alpha \right] = 1$, then define the values $\{p_{i,n}, i \le n, n \in \mathbb{N} \}$ by
\[
	p_{i,n} = \frac{E[ C_i^\alpha |N=n]}{\sum_{j=1}^n E[ C_j^\alpha|N=n]} \in [0,1].
\]
Then,
\begin{align*}
	&\widetilde P( \tilde N=n, \tilde Q \in dy, \tilde C_1\in dx_1,\dots,  \tilde C_n \in dx_n) \\
	&= E\left[1(N=n, Q\in dy, C_1\in dx_1, \dots,  C_n \in dx_n)\sum_{i=1}^N C_i^\alpha\right] \\
	&= P(N=n) \sum_{i=1}^n E[1(C_i \in dx_i )C_i^\alpha| N=n]P(Q\in dy, C_j \in dx_j, j \ne i | C_i = x_i, N=n) \\
	&=: \widetilde P(\tilde N=n)\sum_{i=1}^n p_{i,n}  \tilde f_{i,n} (x_i) P(Q\in dy, C_j \in dx_j, j \ne i | C_i = x_i, N=n),
\end{align*}
where
\[
	\tilde f_{i,n}(x) = \frac{E[1(C_i \in dx) C_i^\alpha | N=n]}{E[ C_i^\alpha|N=n]} = \frac{x^\alpha f_{i,n}(x)}{E[C_i^\alpha | N = n]}
\]
is the tilted marginal density of $\tilde C_i$ conditional on $\tilde N=n$, while $f_{i,n}$ is the marginal density of $C_i$ conditional of $N = n$ under $P$.

Suppose now that $\tilde f_{i,n}$ specifies a 
distribution that can be efficiently simulated, and that it is possible to simulate the vector $(Q, C_1, \dots, C_{i-1}, C_{i+1}, \dots C_n)$ given $\{C_i = x, N = n \}$ under $P$. Then, conditional on $\tilde N = n$, the tilted vector $(\tilde Q, \tilde C_1, \ldots, \tilde C_n)$ can be simulated by picking $i \in \{1, \ldots, n\}$ according to the distribution $\{p_{i,n} : 1\le i \le n\}$, generating $\tilde C_i$ according to $\tilde f_{i,n}$, and then generating $\{\tilde Q, \tilde C_j, j\ne i\}$ according to the conditional distribution of $\{Q,  C_j, j \neq i\}$ given  $\{ C_i, N \}$ under $P$.

Consider the special case when the $\{C_i\}$ are i.i.d. and $N$, $Q$, and $\{C_i\}$ are mutually independent. Then, 
\begin{equation} \label{eq:SizeBias}
	\widetilde P(\tilde N=n) = E\left[1(N=n) \sum_{j=1}^N C_j^\alpha\right] = \frac{nP(N=n)}{E[N]},  \quad n \geq 1,
\end{equation}
since $\alpha$ is such that $E[N]E[C_1^\alpha] = 1$. Hence, under the tilt, $\tilde N$ is the sized-biased version of $N$. Furthermore, $p_{i,n} = 1/n$, $\tilde f_{i,n} = \tilde f$ and $f_{i,n} = f$ for all $i$ and $n$ and some densities $\tilde f$, $f$. So upon simulating $\tilde N$ according to the size-biased distribution, the $\{\tilde C_1, \dots, \tilde C_n\}$ can be simulated by picking $i \in \{1, \ldots, n\}$ uniformly at random, simulating $\tilde C_i$ according to
$$\tilde f(x) = \frac{x^\alpha f(x)}{E[C_1^\alpha]},$$
and simulating the rest of the $\{ \tilde C_j: j \neq i\}$ according to $f$. In this case the distribution of $Q$ is invariant under the tilt. 
\end{example}

Having now described three general methods for simulating the generic branching vector $\boldsymbol{\tilde \psi}$ under the tilt induced by measure $\widetilde{P}$ for nodes along the spine, we now give some more concrete examples that lead to explicit distributions for both $\boldsymbol{\psi}$ and $\boldsymbol{\tilde \psi}$.

\begin{example}\label{Ex:MM1}[The branching version of the $M/M/1$ queue] \upshape As mentioned earlier, the special case of \eqref{eq:HighOrderLindley} when $N \equiv 1$ corresponds to the Lindley equation satisfied by the single-server queue. In particular, if we choose $X = \chi - \tau$ where $\chi$ and $\tau$ are exponentially distributed and independent of each other, we obtain the $M/M/1$ queue. This choice of $X$ is known to be closed under the change of measure induced by $\widetilde P$, in the sense that it remains a difference of two exponentials (but with different rates). As one would expect, this canonical example for the non-branching case is also valid in the branching one. Specifically, suppose that the $\{C_i\}_{i \geq 1}$ are i.i.d.~and independent of $N$, with each of the $C_i = e^{\chi_i - \tau_i}$, where the $\{ (\chi_i, \tau_i) \}_{i \geq 1}$ are i.i.d.~copies of $(\chi, \tau)$, with $\chi$ and $\tau$ exponentially distributed and independent of each other.

Suppose $\tau$ has rate $\lambda$ and $\chi$ has rate $\theta$, for which we have:
$$f(x) =   \frac{\theta \lambda }{\lambda+\theta} \left( x^{\lambda-1}1(x < 1) + x^{-\theta-1} 1(x \geq 1)\right), \qquad x \in (0, \infty),$$
in Example~\ref{E.InversionTransf}. Then, we can simulate $(\tilde N, \tilde C_1, \dots, \tilde C_N)$ under $\widetilde P$ by first simulating $\tilde N$ according to the size-biased distribution of $N$ \eqref{eq:SizeBias}, then pick an index $i \in \{1, \dots, \tilde N\}$ uniformly at random, simulate each of the $\{\tilde C_j: j \neq i\} \stackrel{\mathcal{D}}{=} \{ C_j: j \neq i\}$ through an inversion transform for each of the $\tau_j$ and $\chi_j$, and then simulate $\tilde C_i = e^{\chi_i - \tau_i}$ according to the tilted density given by:
$$\tilde f(x) =  \frac{x^\alpha f(x)}{E[C_1^\alpha]} =   \frac{(\theta-\alpha)(\lambda+\alpha)}{\theta+\lambda} \left( x^{\alpha+\lambda-1} 1(x < 1) + x^{-(\theta-\alpha)-1} 1(x \geq 1) \right) , \qquad x \in (0, \infty),$$
which corresponds to simulating $\chi_i \sim\mbox{Exponential}(\theta-\alpha)$ and $\tau_i\sim\mbox{Exponential}(\lambda+\alpha)$,  independent of each other.
\end{example}

\begin{example}[Identical $C$'s] \label{Ex:EqualCs} \upshape We now give three examples for which $C_i \equiv C$ for all $i \geq 1$.

\begin{enumerate}[(a)]

\item Suppose $Q, N, C$ are mutually independent, $C \sim \mbox{Pareto}(a, b)$ with shape $a$ and scale $b$, and $E[N] < b^{-\alpha}$, where the Cram\'er-Lundberg root $\alpha$ solves $\alpha = (1 - E[N]b^\alpha)a$. Then under $\widetilde P$, $\tilde Q \stackrel{\mathcal{D}}{=} Q$ is invariant and remains independent of $(\tilde N, \tilde C)$, the law of $\tilde N$ is the  sized-biased distribution given by \eqref{eq:SizeBias},
and $\tilde C$ is again Pareto independent of $\tilde N$ and $\tilde Q$, but with shape $a - \alpha$ and scale $b$.

\item Suppose $Q$ is independent of $(N, C)$, $C \sim \mbox{Exponential}(\lambda)$, and conditional on $C$, $N \sim \mbox{Poisson}(C) + 1$.
Then after tilting, the law of $\tilde Q$ remains invariant, $\tilde Q$ remains independent of $\tilde N$ and $\tilde C$, $\tilde N$ has mass function
\[
	\widetilde P(\tilde N = n) = \frac{n\lambda \Gamma(n+\alpha)}{(n-1)! (\lambda+1)^{n+\alpha}}, \quad n \ge 1,
\]
and conditional on $\tilde N$, we have $\tilde C \sim \mbox{Gamma}(\tilde N + \alpha, \lambda + 1)$.

\item Suppose that $Q \sim \mbox{Gamma}(2, \beta)$, with shape $2$ and rate $\beta$, and $N \sim \mbox{Geometric}(1/2)$ with support on $\mathbb{N}_+$, are independent, and conditional on $(N, Q)$, $C \sim \mbox{Gamma}(N + 1, 2Q)$.
Under the tilt, $\tilde Q \sim \mbox{Gamma}(2 - \alpha, \beta)$, conditional on $\tilde Q$, $\tilde C \sim \mbox{Gamma}(\alpha + 2, \tilde Q)$, and conditional on $(\tilde Q, \tilde C)$, $\tilde N \sim \mbox{Poisson}(\tilde Q \tilde C) + 1$.

\end{enumerate}
\end{example}

\begin{example}\label{Ex:ConSimplex}[$C$'s on the $N$-simplex] \label{Ex:Dirichlet} \upshape Let $B \sim\mbox{Gamma}(a,b)$, with shape $a$ and rate $b$, let $N$ have an arbitrary distribution that is independent of $B$, and let $N$ and $(\beta_1, \ldots, \beta_N)$ be such that
\[
	\sum_{i=1}^N \beta_i  = 1.
\]
For example, conditional on $N$, $(\beta_1, \ldots, \beta_N) \sim \mbox{Dirichlet}(\bm{\theta})$ for some concentration parameters $\bm{\theta} = (\theta_1, \ldots, \theta_N)$, i.e., each $\beta_i$ has a marginal $\mbox{Beta}\left(\theta_i, \sum_{k=1}^N \theta_k - \theta_i\right)$ distribution.
Then let $\alpha$ be such that $E[B^\alpha] = 1$, let $C_i = B\beta_i^{1/\alpha}$ for $1\le i \le N$, and let $Q$ be arbitrarily distributed independent of everything else. Then $\alpha$ is the Cram\'er-Lundberg root since
\[
	E\left[\sum_{i=1}^N C_i^\alpha \right] = E\left[ \sum_{i=1}^N B^\alpha \beta_i \right] = E[B^\alpha] = 1.
\]
Under $\widetilde P$ the vector $\boldsymbol{\tilde \psi} = (\tilde N, \tilde Q, \tilde C_1, \tilde C_2, \dots)$ remains in the same family of distributions, i.e., the marginal laws of $\tilde N$ and $\tilde Q$ are invariant, $B$ is tilted to $\tilde B \sim \mbox{Gamma}(a + \alpha, b)$, and the $\tilde C_i$ are constructed in the same way using the same $\beta_i$. In this case, it is the particular dependence of $N$ on $\{C_i\}$ that ensures the invariance of $\tilde N$, since
\[
	\widetilde P(\tilde N=n) = E\left[1(N=n) \sum_{i=1}^N C_i^\alpha \right] = E[1(N=n)B^\alpha] = P(N=n).
\]

\end{example}

\begin{example}[Constant $N$ and discrete $C$'s]\rm
The case when $N\equiv n_0$ is constant and the $C_i$, $i=1,\ldots,n_0$,  take discrete values is even simpler. Note that
in this case
\[\widetilde P( \tilde C_1=c_1, \ldots, \tilde C_{n_0}=c_{n_0})=P(C_1=c_1, \ldots, C_{n_0}=c_{n_0})(c_1^\alpha+\ldots +c_{n_0}^\alpha).\]
For example, taking $N \equiv 2$ and
$P\left((  C_1,C_2) = \left(\frac{2}{3},0\right)\right)=\frac{3}{4}=1-P((C_1,C_2) = (1,1))$
we have 
$\alpha = 1$ and 
\[\widetilde P\left(( \tilde C_1, \tilde C_2) = \left(\tfrac{2}{3},0\right)\right)=\frac{1}{2}=\widetilde P\left((\tilde C_1,\tilde C_2) = (1,1)\right).\]
\end{example}

\subsection{Numerical experiments}
\begin{table}[t]
\caption{Numerical results for the branching M/M/1 queue, sample size 10,000}
\begin{center}
\begin{tabular}{|rrrrrrr|} \hline 
\multicolumn{7}{|l|}{Branching $M/M/1$ queue: $\alpha = 4.374$, $\mu = 1.383$} \\
$t$ & $\qquad\bar Z(t)$ & \qquad Std. Err. & $\qquad t/\mu$ & \begin{tabular}{@{}c@{}}Terminal gen.  \end{tabular} & Time & \begin{tabular}{@{}c@{}}Prop. nonzero \end{tabular}  \\ \hline
     0.5    &      0.037774 &    0.001241 &    0.36 &      1.39   &       0.002610 &      0.967 \\   
         1    &    0.003025 &    0.000123 &    0.72 &      1.78      &    0.007702 &      0.980 \\   
      1.5    &   0.000354 &    1.07147e-05 &     1.08 &      2.16   &        0.017536 &      0.983 \\   
       2    &    3.90110e-05  &  1.43477e-06   &  1.45 &      2.52      &    0.029310 &      0.983 \\   
      2.5    &    4.11873e-06  &   1.16323e-07  &    1.81 &     2.90 &       0.065747 &      0.985 \\ \hline   
\end{tabular}
\label{T.MM1}
\end{center}
\end{table}
Here we implement two examples of the importance sampling algorithm. The first is the branching M/M/1 queue of Example~\ref{Ex:MM1}, in which we let $\chi$ have rate $5$ and $\tau$ have rate $1/4$, and we let $N$ be a truncated Poisson random variable with mean $2$, i.e. $N \overset{\mathcal{D}}{=} K|K>0$, where $K \sim \mbox{Poisson}(2)$. 
In this case $\alpha = 4.374$, and we include a perturbation $Y \sim \mbox{Exponential}(9)$ independent of $(N, \{C_i\})$, so that $Q$ is a Pareto random variable with enough moments to ensure $E[Q^{2\alpha}] < \infty$ and provide our estimator with bounded relative error (see Lemma~\ref{L.BddRelError}). Under the tilt, $N$ has its size-biased distribution $\mbox{Poisson}(2) + 1$, and the $\{C_i\}$ are simulated as described in Example~\ref{Ex:Mixture}, by picking one uniformly at random and applying an exponential tilt. Since $Q$ is independent of $(N, \{C_i\})$, we use the estimator in \eqref{eq:IndepQEstimator}
(See Remarks~\ref{R.MainRemarks}(a) and \ref{R.ReduceToSiegmund}). 
In Table~\ref{T.MM1} we show the numerical results, which include for a range of $t$ values the sample average $\bar Z(t)$ based on 10,000 copies of $Z(t)$ and the standard error in the estimate. Additionally, we give the average tree generation $\tau(t) + 1$ at which the algorithm terminates, the value of $t/\mu$ for comparison, the average time in seconds to generate one copy of $Z(t)$, and the fraction of the estimates that are nonzero. 

\begin{table}[b]
\caption{Numerical results for $C$'s on the $N$-simplex, sample size 10,000}
\begin{center}
\begin{tabular}{|rrrrrrr|} \hline
\multicolumn{7}{|l|}{$C$'s on the $N$-simplex: $\alpha = 3.328$, $\mu = 0.995$} \\
$t$ & $\qquad\bar Z(t)$ & \qquad Std. Err. & $\qquad t/\mu$ & \begin{tabular}{@{}c@{}}Terminal gen.  \end{tabular} & Time & \begin{tabular}{@{}c@{}}Prop. nonzero \end{tabular}  \\ \hline
      1.5   &       0.015785 &    0.000166 &    1.51  &    0.33   &   0.000235    &  0.998   \\
        2  &      0.003611 &   3.51666e-05  &   2.01  &    0.78   &        0.000311  &    0.994   \\
      2.5  &       0.000613 &   6.60663e-06  &   2.51  &    1.33  &    0.000439   &   0.994   \\
        3  &       0.000116 & 1.21042e-06   &  3.01   &   1.84  &    0.000671   &   0.994     \\
      3.5   &    2.29240e-05 &    2.35959e-07  &   3.52   &   2.33 &              0.001058   &  0.992  \\ \hline
\end{tabular}
\label{T.Simplex}
\end{center}
\end{table}

Figure~\ref{F.LogprobplotsMM1} shows a plot of $\log \bar Z(t)$ compared with the tail asymptotic $\log(He^{-\alpha t})$ over the range of values in the table, where $H$ is computed using the population dynamics algorithm \cite{Olv_19}. As can be seen, the distribution becomes indistinguishable from the tail asymptotic somewhere in the middle of this range. The terminal generation of each estimator does not converge to $t/\mu$ as quickly; the terminal generations listed are greater than $t/\mu$ despite the perturbation $Q \ge 1$ which in this case can only cause the process $S_{\bfi} + Y_{\bfi}$ to reach the level $t$ earlier than the random walk $S_{\bfi}$.
The large fraction of estimates that are nonzero in this example indicates that in almost all iterations, the level $t$ was first reached on the spine of the tree. In the case of i.i.d. $C$'s, only one offspring of each node on the spine is chosen for the tilt, and this titled branch is then the most likely to be the next step in the spine, making the event that a path off the spine hits $t$ first unlikely. 

\begin{figure}[t]
	\centering 
	\begin{subfigure}[t]{0.45\textwidth}
		\centering 
		\begin{picture}(200,200)(0,0)
			\put(3,10){\includegraphics[scale=0.585]{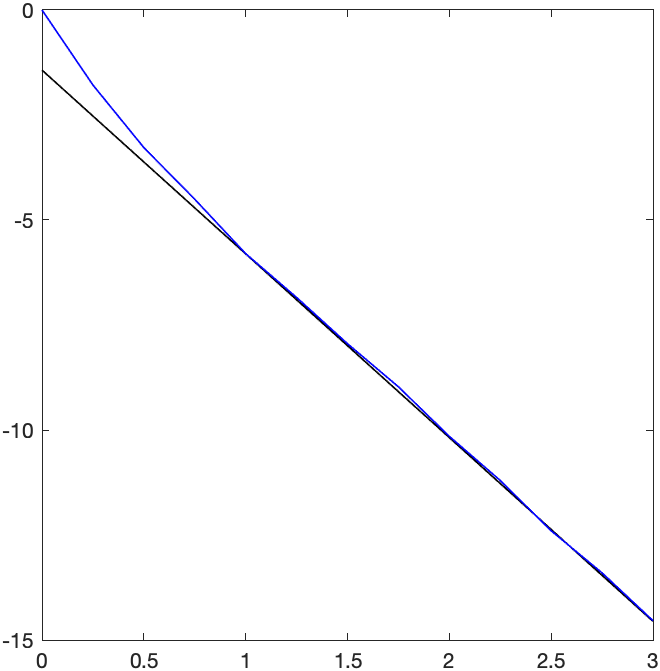}}
			\put(-6,80){\rotatebox{90}{\footnotesize $\log P(W>t)$}}
			\put(100,-0){\footnotesize $t$}
		\end{picture}
		\caption{Log probabilities for the branching $M/M/1$ queue, $H = 0.2390$.}
		\label{F.LogprobplotsMM1}
	\end{subfigure}%
	~
	\begin{subfigure}[t]{0.45\textwidth}
		\centering
		\begin{picture}(200,200)(0,0)
			\put(3,10){\includegraphics[scale=0.585]{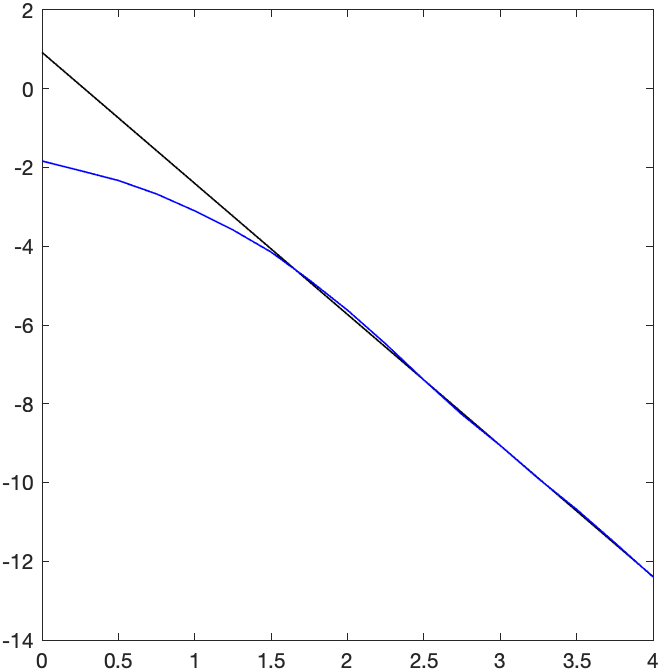}}
			\put(-6,80){\rotatebox{90}{\footnotesize $\log P(W>t)$}}
			\put(100,-0){\footnotesize $t$}
		\end{picture}
		\caption{Log probabilities for $C$'s on the $N$-simplex, $H = 2.5180$.}
		\label{F.LogprobplotsSimplex}
	\end{subfigure}
	\caption{Log probability plots estimated two ways: black lines are the logs of asymptotic approximations $H e^{-\alpha t}$ and blue lines are the logs of the estimates $\bar Z(t)$.}
\end{figure}
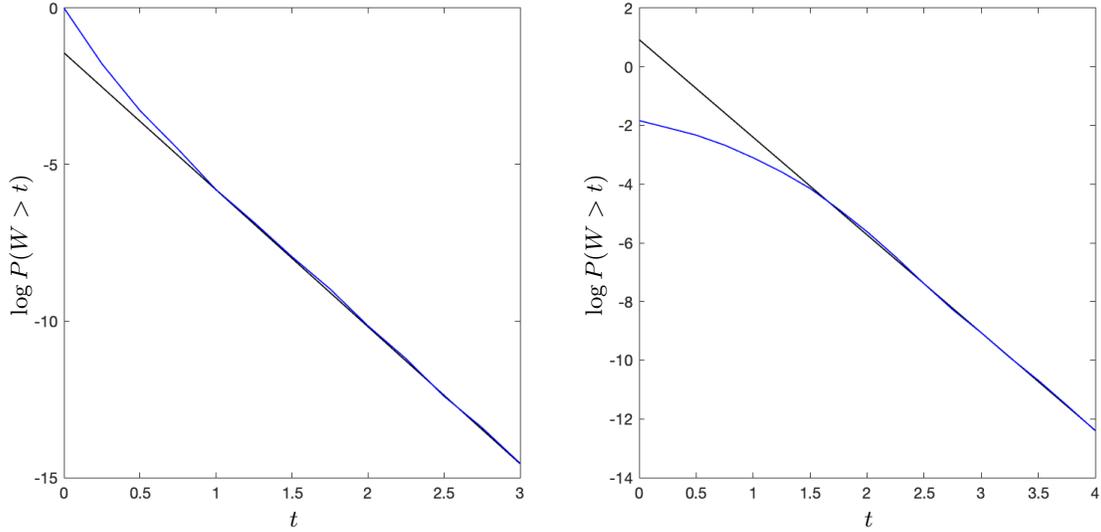

The second experiment is for $C$'s on the $N$-simplex, as in Example~\ref{Ex:ConSimplex}. We choose $B \sim \mbox{Gamma}(1/4, 1)$, 
set $\alpha$ so that $E[B^\alpha] = 1$, and then let $C_i = B\beta_i^{1\alpha}$, where the $\{\beta_i : 1\le i \le N\}$ are $\mbox{Dirichlet}(1,\ldots,1)$ random variables conditional on $N$. We let $N$ be uniform over $\{1,2,3\}$ independently of $B$, and we take $Q = 2B$, so that the perturbation is positively correlated with the $\{C_i\}$.
As noted in Example~\ref{Ex:ConSimplex}, the change of measure tilts $B$, the distribution of $N$ is the same, and $(Q, \{C_i\})$ is generated the same way through $B$ and $\{\beta_i\}$. Because of the dependence between $Q$ and $\{C_i\}$, the estimator we use is in its most general form \eqref{eq:OurEstimator}. 
Table~\ref{T.Simplex} gives the numerical results for a range of $t$ values and a sample size of 10,000 for each, and Figure~\ref{F.LogprobplotsSimplex} is a plot of $\log \bar Z(t)$ and $\log (He^{-\alpha t})$, where again $H$ is obtained using the population dynamics algorithm. 

Again in this experiment, almost every estimator terminated on the spine of the tree. Unlike in the i.i.d. $\{C_i\}$ case, the tilted $B$ influences the $C$'s of every offspring of a node on the spine, so one might expect a significant percentage of estimators to terminate on nodes off the spine but which have the spine in their recent ancestry. We do not see this likely because $Q = 2B$ also is made larger by the tilt but only on the spine, and it seems that it is the perturbation $Y_{\bfi}$ which is causing the process $S_{\bfi} + Y_{\bfi}$ to reach the level $t$ in almost every iteration. This tilted perturbation is likely also causing the terminal tree generation to be smaller than $t/ \mu$ for these values of $t$, as it is greater than zero with high probability.

In each example, there is a range of $t$ values in the tail of the distribution for which $P(W>t)$ is not yet indistinguishable from the asymptotic behavior. 
The importance sampling algorithm presented here provides an efficient method of simulating tail probabilities for these intermediate values of $t$ before the asymptotic behavior dominates.

\section{Proofs}\label{S.Proofs}

In this section we provide the proofs to all of our results. To ease its reading, we start first with the proofs of Lemmas~\ref{L.Newmeasure} and \ref{L.DistOnJ}, which describe the distribution of the tree $\mathcal{T}$ under $\widetilde P$. We then give the proof of our main theoretical result, Theorem~\ref{T.Main}, followed  by the proofs of Lemmas~\ref{L.Tovermu} and \ref{L.BddRelError} which are related to our importance sampling estimator. Finally, we end the paper with a short proof of Theorem~1 in \cite{Araman} for the non-branching, bounded $Q$ case (Theorem~\ref{T.BoundedQNonBranching}). Throughout the remainder of the paper we assume that Assumption~\ref{Ass1} holds for some $\alpha > 0$.

\subsection{The distribution of $\mathcal{T}$ under $\widetilde{P}$}

We start with the proof of Lemma~\ref{L.Newmeasure}, which provides the distribution of the generic branching vectors defining the weighted tree $\mathcal{T}$. The distribution for vectors on the spine is different under $\widetilde P$ and $P$, whereas that of vectors off the spine remains the same.

\begin{proof}[Proof of Lemma \ref{L.Newmeasure}]	
We will start by deriving an expression for the joint distribution of the vectors $\bm{\psi}_{\bfi}$ along then spine. To do this, fix ${\bfi} \in \mathbb{N}_+^k$ and let $B_0, B_1, \dots, B_k \subseteq \mathbb{N} \times \mathbb{R}^\infty$ be measurable sets. Next, note that the event $\{ \bm{\psi}_{{\bfi} |r} \in B_r\,, r= 0,\ldots, k; \,{\bfJ}_k ={\bfi} \}$ is measurable with respect to $\mathcal{G}_{k+1}$, and therefore,
\begin{align*}
&\widetilde{P}\left( \bm{\psi}_{{\bfi} |r} \in B_r\,, r= 0,\ldots, k; \,{\bfJ}_k ={\bfi}  \right)  \\
&= E\left[ 1\left( \bm{\psi}_{{\bfi} |r} \in B_r\,, r= 0,\ldots, k; \,{\bfJ}_k ={\bfi}  \right) L_{k+1} \right] \\
&= E\left[ E\left[ \left. 1\left( \bm{\psi}_{{\bfi} |r} \in B_r\,, r= 0,\ldots, k; \,{\bfJ}_k ={\bfi}  \right) L_{k+1} \right| \mathcal{G}_k \right] \right] \\
&= E\left[ 1\left( \bm{\psi}_{{\bfi} |r} \in B_r\,, r= 0,\ldots, k-1; \,{\bfJ}_{k} = {\bfi}  \right)  L_{k}  E\left[ \left. 1\left( \bm{\psi}_{\bfi} \in B_k  \right) D_{\bfi}  \right| \mathcal{G}_{k} \right] \right] \\
&= E\left[ 1\left( \bm{\psi}_{{\bfi} |r} \in B_r\,, r= 0,\ldots, k-1; \,{\bfJ}_{k} = {\bfi}  \right)  L_{k}   \right] E\left[ 1(\bm{\psi} \in B_k) D \right] \\
&= E\left[ 1\left( \bm{\psi}_{{\bfi} |r} \in B_r\,, r= 0,\ldots, k-2; \,{\bfJ}_{k-1} = ({\bfi}|k-1)  \right) L_{k-1} \right. \\
&\hspace{5mm} \left. \times E\left[ \left. 1\left( \bm{\psi}_{{\bfi} |k-1} \in B_{k-1}; \,{\bfJ}_{k} = {\bfi}  \right)  D_{{\bfJ}_{k-1}}  \right| \mathcal{G}_{k-1} \right]  \right]  E\left[ 1(\bm{\psi} \in B_k) D \right].
\end{align*}
Now note that by letting ${\bfj} = ({\bfi}|k-1)$ we obtain
\begin{align*}
&E\left[ \left. 1\left( \bm{\psi}_{{\bfi}|k-1} \in B_{k-1}; \,{\bfJ}_k ={\bfi}  \right) D_{{\bfJ}_{k-1}}  \right| \mathcal{G}_{k-1} \right]  \\
&= E\left[ 1( \bm{\psi}_{\bfj} \in B_{k-1}; \text{ offspring $i_k$ of ${\bfj}$ is chosen}) D_{\bfj} \right] \\
&= E\left[ 1( \bm{\psi}_{\bfj} \in B_{k-1}, N_{\bfj} \geq i_k) \cdot \frac{C_{({\bfj}, i_k)}^\alpha}{D_{\bfj}}  \cdot D_{\bfj} \right]  \\
&= E\left[ 1(\bm{\psi} \in B_{k-1}, \, N \geq i_{k}) C_{i_k}^\alpha \right].
\end{align*}
It follows that
\begin{align*}
&E\left[ 1\left( \bm{\psi}_{{\bfi} |r} \in B_r\,, r= 0,\ldots, k-1; \,{\bfJ}_{k} = {\bfi}  \right)  L_{k}   \right]  \\
&= E\left[ 1\left( \bm{\psi}_{{\bfi} |r} \in B_r\,, r= 0,\ldots, k-2; \,{\bfJ}_{k-1} = ({\bfi}|k-1)  \right)  L_{k-1} \right]  E\left[ 1(\bm{\psi} \in B_{k-1}, \, N \geq i_{k}) C_{i_k}^\alpha \right]   \\
&= E\left[ 1\left( \bm{\psi}_\emptyset \in B_0;  \,{\bfJ}_{1} = i_1  \right)  L_{1} \right]  \prod_{r=2}^{k} E\left[ 1(\bm{\psi} \in B_{r-1}, \, N \geq i_{r}) C_{i_r}^\alpha \right]  \\
&= \prod_{r=1}^{k} E\left[ 1(\bm{\psi} \in B_r, \, N \geq i_{r}) C_{i_r}^\alpha \right].
\end{align*}
We conclude that
\begin{equation} \label{eq:NewMeasure}
\widetilde{P}\left( \bm{\psi}_{{\bfi} |r} \in B_r\,, r= 0,\ldots, k; \,{\bfJ}_k ={\bfi}  \right)  = E\left[ 1(\bm{\psi} \in B_k) D \right] \prod_{r=1}^{k} E\left[ 1(\bm{\psi} \in B_{r-1}, \, N \geq i_{r}) C_{i_r}^\alpha \right].
\end{equation}

In particular, by setting $B_r = \mathbb{N} \times \mathbb{R}^\infty$ for all $r = 0, 1, \dots, k$, we obtain the first expression in the statement of the lemma, i.e.,
\begin{equation*}
\widetilde{P}( {\bfJ}_k = {\bfi} ) = \prod_{r=1}^{k} E[ 1(N \geq i_{r}) C_{i_r}^\alpha ].
\end{equation*}
Similarly, by setting $B_k = B$ and $B_r = \mathbb{N} \times \mathbb{R}^\infty$ for all $r = 0, 1, \dots, k-1$, we obtain the third expression:
\begin{align*}
\widetilde{P}( \bm{\psi}_{\bfi} \in B | {\bfi} = {\bfJ}_k) &= E[1(\bm{\psi} \in B) D] = E\left[ 1((N, Q, C_1, C_2, \dots) \in B) \sum_{j=1}^N C_j^\alpha \right].
\end{align*}

To obtain the corresponding expression for nodes off the spine (the second expression in the statement of the lemma) note that the same conditioning approach used for a node on the spine gives, for any ${\bfi} \in \mathbb{N}_+^k$ and any measurable $B \subseteq \mathbb{N} \times \mathbb{R}^\infty$,
\begin{align*}
\widetilde{P}\left( \bm{\psi}_{\bfi} \in B; {\bfJ}_k \neq {\bfi} \right) &= E\left[ E\left[ \left. 1\left( \bm{\psi}_{\bfi} \in B; {\bfJ}_k \neq {\bfi} \right) L_{k+1} \right| \mathcal{G}_k \right] \right] \\
&= E\left[ L_k E\left[ \left. 1\left( \bm{\psi}_{\bfi} \in B; {\bfJ}_k \neq {\bfi} \right) D_{{\bfJ}_k}  \right| \mathcal{G}_k \right] \right] \\
&= E\left[ L_k E[ 1(\bm{\psi}_{\bfi} \in B) | \mathcal{G}_k] E\left[ \left. 1\left( {\bfJ}_k \neq {\bfi} \right) D_{{\bfJ}_k}  \right| \mathcal{G}_k \right] \right] \\
&= P(\bm{\psi} \in B) E\left[ 1({\bfJ}_k \neq {\bfi}) L_{k+1} \right]  \\
&= P(\bm{\psi} \in B) \widetilde{P}( {\bfJ}_k \neq {\bfi}).
\end{align*}
Therefore,
$$\widetilde{P}\left( \bm{\psi}_{\bfi} \in B | {\bfJ}_k \neq {\bfi} \right) = P(\bm{\psi} \in B) = P((N, Q, C_1, C_2, \dots) \in B).$$
The conditional independence of the vectors $\{ \bm{\psi}_{\bfi}: {\bfi} \in A_k\}$ given $\mathcal{G}_{k-1}$ follows from the branching property under $P$. This completes the proof.
\hfill\end{proof}

We now give the proof of Lemma~\ref{L.DistOnJ}, which gives the distribution of the random walk defined by the nodes along the spine.

\begin{proof}[Proof of Lemma \ref{L.DistOnJ}]
Recall that $\hat X_k = X_{{\bfJ}_k} = \log C_{{\bfJ}_k}$ and $\xi_k = Y_{{\bfJ}_k} = \log Q_{{\bfJ}_k}$. By conditioning on all possible paths that could be chosen to define ${\bfJ}_k$ we obtain
\begin{align*}
&\widetilde P\left( \hat X_1 \leq x_1, \dots, \hat X_k \leq x_k, \, \xi_k \leq y \right) \\
&= \sum_{{\bfi} \in \mathbb{N}_+^k} \widetilde P\left( \hat X_1 \leq x_1, \dots, \hat X_k \leq x_k, \, \xi_k \leq y, \, {\bfJ}_k = {\bfi}\right ) \\
&= \sum_{{\bfi} \in \mathbb{N}_+^k} \widetilde P\left( C_{{\bfi}|1} \leq e^{x_1}, \dots, C_{{\bfi}|k}  \leq e^{x_k}, \, Q_{{\bfi}|k} \leq e^y, \, {\bfJ}_k = {\bfi} \right) \\
&= \sum_{{\bfi} \in \mathbb{N}_+^k} \widetilde P(Q \leq e^y) \prod_{r=1}^{k} E\left[ 1( C_{i_r} \leq e^{x_r}, N \geq i_r) C_{i_r}^\alpha \right] \\
&= \widetilde P(Q \leq e^y)  \sum_{i_1 =1}^\infty E\left[ 1( C_{i_1} \leq e^{x_1}, \, N \geq i_1) C_{i_1}^\alpha \right] \cdots \sum_{i_k =1}^\infty  E\left[ 1( C_{i_k} \leq e^{x_k}, \, N \geq i_k) C_{i_k}^\alpha \right] \\
&= E\left[ 1(Q \leq e^y) \sum_{i=1}^N C_i^\alpha \right] \prod_{r=1}^k G(x_r),
\end{align*}
where in the third equality we used \eqref{eq:NewMeasure} and the independence of $Q_{{\bfi}|k}$ and $\bm{\psi}_{{\bfi}|k-1}$. To compute the mean of the $\hat X_i$'s 
note that
\begin{align*}
\widetilde{E}\left[ |\hat X_1| \right] &= \int_{-\infty}^\infty |x| G(dx) = \int_{-\infty}^\infty |x| E\left[ \sum_{i=1}^N 1(\log C_i \in dx) C_i^\alpha \right] = E\left[ \sum_{i=1}^N C_i^\alpha |\log C_i|  \right].
\end{align*}
Now choose $0 < \beta < \alpha$ such that $E\left[ \sum_{i=1}^N C_i^\beta \right] < 1$ and note that since $\log^- C_i = 0$ when $C_i > 1$, 
\begin{align*}
\widetilde{E}\left[ \hat X_1^- \right] &= E\left[ \sum_{i=1}^N C_i^\alpha \log^- C_i  \right] \leq \sup_{0 \leq x \leq 1} x^{\alpha-\beta} |\log x| \widetilde{E}\left[ \sum_{i=1}^N C_i^\beta\right] < \infty.
\end{align*}
For the positive part note that by Assumption~\ref{Ass1}(a) we have $E\left[ \sum_{i=1}^N C_i^\alpha \log C_i \right] \in (0,\infty)$, and therefore
$$\widetilde{E}\left[ \hat X_1^+ \right] = E\left[ \sum_{i=1}^N C_i^\alpha \log C_i \right] + \widetilde{E}\left[ \hat X_1^- \right] < \infty.$$
Finally, since both $\widetilde{E}\left[ \hat X_1^- \right]$ and $\widetilde{E}\left[ \hat X_1^+ \right]$ are finite, we have $\widetilde{E}\left[ \hat X_1 \right] = E\left[ \sum_{i=1}^N C_i^\alpha \log C_i \right]  \in (0, \infty)$. 
\hfill\end{proof}

\subsection{Proof of Theorem~\ref{T.Main}}

We now move on to the proof of our main theorem, which is obtained by applying renewal theory to compute the limit of
$$\widetilde{E}\left[ 1({\bfJ}_{\tau(t)} = \gamma(t)) e^{-\alpha(V_{\tau(t)} - t)}D_{{\bfJ}_{\tau(t)}}^{-1}  \right].$$
The proof will rely on several preliminary results, the first of which establishes the almost sure finiteness of $\tau(t)$ under $\widetilde{P}$.  Throughout this subsection we assume all parts of Assumption~\ref{Ass1} hold.

\begin{lemma} \label{L.nuFinite}
For any $y\in \mathbb{R}$, $\tau(y) <\infty$ $\widetilde P$-a.s.
\end{lemma}

\begin{proof}
The sequence $\{V_k\}$ satisfies the strong law of large numbers
\[
 V_n / n \rightarrow  \mu >0 \quad \widetilde P \mbox{-a.s.}
 \]
Note that Assumption~\ref{Ass1}(c) ensures that $\xi_0$ has the same support under $\widetilde{P}$ that under $P$, hence, since $P(Q > 0) > 0$, there must exist an $\epsilon > 0$ such that $P(Q > \epsilon) > 0$, and therefore, $\widetilde{P}(Q > \epsilon) > 0$. Hence, under $\widetilde{P}$, the random times $T_0 = \inf\{ i \geq 0: \xi_i > \log \epsilon\}$ and $T_{k+1} = \inf\{i > T_k: \xi_i > \log \epsilon\}$ are finite $\widetilde{P}$-a.s.~for all $k \geq 0$; moreover, $T_k \to +\infty$ $\widetilde{P}$-a.s.~as $k \to \infty$. Focusing on subsequences along the indexes $\{T_k : k \geq 0\}$ gives
\begin{align*}
\liminf_{k \to \infty} \frac{V_{T_k} + \xi_{T_k}}{T_k} \geq \liminf_{k \to \infty} \frac{V_{T_k} +\log \epsilon}{T_k} = \mu \qquad \widetilde{P}\text{-a.s}
\end{align*}
Since 
$$\sup_{n \geq 0}\,  (V_n + \xi_n ) \geq \sup_{k \geq 0} \, ( V_{T_k} + \xi_{T_k} ) \qquad \widetilde{P}\text{-a.s.},$$
and $\widetilde{P}\left( \sup_{k \geq 0} ( V_{T_k} + \xi_{T_k} ) > y \right) = 1$ for all $y$, the result follows. 
\hfill\end{proof}

The next thing we need to establish is an upper bound for $\widetilde{E}\left[ e^{-\alpha(V_{\tau(t)} -t )} D_{{\bfJ}_{\tau(t)}}^{-1} \right]$, since this quantity will appear in various places throughout the proof of Theorem~\ref{T.Main}.

\begin{lemma} \label{L.Overshoot}
Under Assumption~\ref{Ass1}, we have for any $t \in \mathbb{R}$,
\begin{equation}\label{eq:BoundU}
\widetilde{E}\left[ e^{-\alpha(V_{\tau(t)} - t)} D_{{\bfJ}_{\tau(t)}}^{-1} \right] \leq \sum_{n=0}^\infty \widetilde{E}\left[ u(t-V_n) \right],
\end{equation}
where $u(x)  = e^{\alpha x} P\left( Y > x  \right)$ is directly Riemann integrable (d.R.i.). Moreover,
\begin{equation}\label{eq:FiniteLimsup}
\limsup_{t \to \infty} \widetilde{E}\left[ e^{-\alpha(V_{\tau(t)} - t)} D_{{\bfJ}_{\tau(t)}}^{-1} \right] \leq \frac{ E[Q^\alpha]}{\alpha \mu}.
\end{equation}
\end{lemma}

\begin{proof}
Note that
\begin{align*}
&\widetilde{E}\left[ e^{-\alpha(V_{ \tau(t)} - t)} D_{{\bfJ}_{\tau(t)}}^{-1} \right] \\
&= \widetilde{E}\left[ e^{\alpha t} D_{{\bfJ}_0}^{-1} 1(\xi_0 > t) \right] + \sum_{n=1}^\infty \widetilde{E}\left[ 1\left( \max_{0 \leq k \leq n-1} V_k + \xi_k \leq t < V_n + \xi_n  \right) e^{-\alpha (V_n - t)} D_{{\bfJ}_n}^{-1} \right] \\
&= \widetilde{E}\left[ e^{\alpha t} D_{{\bfJ}_0}^{-1} 1(\xi_0 > t) \right] \\
&\quad +  \sum_{n=1}^\infty \widetilde{E}\left[ 1\left( \max_{0 \leq k \leq n-1} V_k + \xi_k  \leq t \right) e^{-\alpha(V_n-t)} \widetilde{E}\left[ \left. 1\left( V_n + \xi_n > t  \right) D_{{\bfJ}_n}^{-1} \right| \mathcal{G}_n \right]  \right] \\
&\leq \sum_{n=0}^\infty \widetilde{E} \left[ u(t-V_n) \right],
\end{align*}
where 
$$u(x) = e^{\alpha x} \widetilde{E}\left[ 1(\xi_0 > x) D_{{\bfJ}_0}^{-1} \right] = e^{\alpha x} P(Y > x).$$
We will now show that $u$ is d.R.i.~on $(-\infty, \infty)$, and we start by proving that $u$ is integrable. To see this note that 
\[
\int_{-\infty}^\infty u(x) dx = E\left[ \int_{-\infty}^\infty e^{\alpha x} 1(Y > x)\,dx \right]
= \frac{E[Q^\alpha]}{\alpha} < \infty.
\]
Now note that for any $h > 0$,
\begin{align*}
\sum_{n=-\infty}^\infty \sup_{y \in (nh, (n+1)h]} u(y) &\leq \sum_{n=-\infty}^\infty  e^{\alpha (n+1)h} P(Y > nh) \\
&\leq \sum_{n=-\infty}^\infty \int_{(n-1)h}^{nh} e^{2\alpha h} e^{\alpha x} P(Y > x) dx \\
&= e^{2\alpha h} \int_{-\infty}^\infty u(x) dx < \infty,
\end{align*}
so by Proposition~4.1(ii) in Chapter V of \cite{Asm2003}, $u$ is d.R.i.

To complete the proof use the two-sided renewal theorem (see Theorem~4.2 in \cite{Ath_McD_Ney_78}), to obtain
\begin{align*}
\lim_{t \to \infty} \sum_{n=0}^\infty \widetilde{E}\left[ u(t-V_n) \right] &= \frac{1}{\mu} \int_{-\infty}^\infty u(x) dx = \frac{ E[Q^\alpha]}{\alpha \mu}. 
\end{align*}
\hfill\end{proof}

\begin{cor}\label{L.unifbnd}
There exists a constant $0 < B < \infty$ such that for any $t \in \mathbb{R}$, 
$$e^{\alpha t} P(W > t) \leq B \, e^{-\alpha (-t)^+}.$$
\end{cor}

\begin{proof} By \eqref{eq:FiniteLimsup}, there is a $t_0 > 0$ such that $\sup_{t\ge t_0} \widetilde E \left[ e^{-\alpha (V_{\tau(t)}-t)} D^{-1}_{{\bfJ}_{\tau(t)}}\right] \leq 2E[Q^\alpha]/(\alpha \mu) < \infty$. Since for $t \leq 0$ we have the trivial bound $e^{\alpha t} P(W > t) \leq e^{-\alpha(-t)^+}$, and for $t \geq t_0$ we have $e^{\alpha t} P(W > t) \leq \widetilde E \left[ e^{-\alpha (V_{\tau(t)}-t)} D^{-1}_{{\bfJ}_{\tau(t)}}\right] \leq 2E[Q^\alpha]/(\alpha\mu)$, we can take $B = \max\left\{e^{\alpha t_0}, 2E[Q^\alpha]/(\alpha\mu) \right\}$ to obtain the stated inequality. 
\hfill\end{proof}

We are now ready to move on to the application of the Markov renewal theorem.

\subsubsection{The Markov Renewal Theorem}

Let $\mathcal{S}^{(m)}$ denote the state space of weighted trees of height $m$ having a path identified as its spine, and define  the Markov chain $\{M_n^{(m)} : n\ge m \}$ in $\mathcal{S}^{(m)}$ as follows:
$$M_m^{(m)} = \left\{ \boldsymbol{\psi}_{\bfi}: {\bfi} \in \bigcup_{k=0}^{m-1} A_k \right\} \cup \left\{ {\bfJ}_0, {\bfJ}_1, \dots, {\bfJ}_m \right\},$$
and 
$$M_{m+n}^{(m)} = \left\{ \boldsymbol{\psi}_{\bfi}: {\bfi} \in \mathcal{T}_{n}^{(m)} \right\} \cup \left\{ {\bfJ}_n, {\bfJ}_{n+1}, \dots, {\bfJ}_{n+m} \right\}$$
for $n > 0$. Recall that 
$$\mathcal{T}_n^{(m)} = \bigcup_{k=0}^{m-1} A_{k,{\bfJ}_n}\qquad \text{and} \qquad A_{k,{\bfi}} = \{ ({\bfi}, {\bfj}) \in \mathcal{T}: |{\bfj}| = k\}.$$

Note that with this notation we can write
\begin{align*}
&\widetilde{E}\left[ 1({\bfJ}_{\tau(t)} = \gamma(t)) e^{-\alpha(V_{\tau(t)} - t)} D_{{\bfJ}_{\tau(t)}}^{-1}  \right] \\
&= \sum_{n=0}^{m-2} \widetilde{E}\left[ 1(|\gamma(t)| = n, {\bfJ}_{\tau(t)} = \gamma(t)) e^{-\alpha(V_{\tau(t)} - t)} D_{{\bfJ}_{\tau(t)}}^{-1}  \right] \\
&\hspace{5mm} + \widetilde{E}\left[ 1(|\gamma(t)| \geq m-1, {\bfJ}_{\tau(t)} = \gamma(t)) e^{-\alpha(V_{\tau(t)} - t)} D_{{\bfJ}_{\tau(t)}}^{-1}  \right] \\
&=  \sum_{n=0}^{m-2} \widetilde{E}\left[ 1( {\bfJ}_{n} = \gamma(t)) e^{-\alpha(V_{n} - t)} D_{{\bfJ}_{n}}^{-1}  \right]  + \widetilde{E}\left[ K\left( M_m^{(m)}, t \right)  \right] ,
\end{align*}
where
$$K\left(M_m^{(m)}, t\right) := \widetilde{E}\left[  \left. 1(|\gamma(t)| \geq m-1, \gamma(t) = {\bfJ}_{\tau(t)} ) e^{-\alpha(V_{\tau(t)} - t)} D_{{\bfJ}_{\tau(t)}}^{-1} \right| M_m^{(m)} \right].$$

The key idea behind the proof of Theorem~\ref{T.Main} is that $\widetilde{E}\left[ K\left( M_m^{(m)}, t \right)  \right]$ can be analyzed using the Markov renewal theorem (Theorem~2.1 in \cite{Alsmeyer_94}). However, the use of this theorem is not immediate, since, as mentioned earlier, the perturbations make it difficult to identify clear regeneration points. In comparison, when the perturbations are not random (i.e., $Q$'s are constant), it suffices to focus on the generations where the ladder heights of the random walk $\{V_k: k \geq 0\}$ occur, since the crossing of level $t$ can only happen at these times. To solve this problem, our approach relies on the observation that although the crossing of level $t$ does not need to coincide with a ladder height of $\{V_k: k \geq 0\}$, and the perturbed (branching) random walk does not regenerate when the ladder heights of $\{V_k+\xi_k: k \geq 0\}$ occur, we can ignore the effect of the perturbations by looking at a long enough stretch of the history of the branching random walk along its spine. This history is what the Markov chain $\{M_{m+n}^{(m)}: n \geq 0\}$ includes. 

Our first technical result in this section will define a function that will appear in the derivation of a lower bound for $\widetilde{E}\left[ K\left( M_m^{(m)}, t \right)  \right]$. Before we state it, we will need to define the following random variables. 
We use
\begin{equation} \label{eq:Subtrees}
W_{\bfi} := \bigvee_{r =0}^\infty \bigvee_{({\bfi}, {\bfj}) \in A_{|{\bfi}|+r}} \left( S_{({\bfi}, {\bfj})} - S_{\bfi}  + Y_{({\bfi},{\bfj})} \right), \qquad {\bfi} \in \mathcal{T},
\end{equation}
to define the maximum of the perturbed branching random walk rooted at node ${\bfi}$. Note that if ${\bfi}$ is not part of the spine, then $W_{\bfi}$ has the same distribution under both $\widetilde{P}$ and $P$. Now use the $W_{\bfi}$ to define
$$Z_k := \xi_k \vee \max_{({\bfJ}_k,i) \neq {\bfJ}_{k+1}} \left( S_{({\bfJ}_k, i)} - V_k + W_{({\bfJ}_k, i)} \right), \qquad k \geq 0.$$
Note that the $Z_k$ is the maximum of the perturbation at the spine node ${\bfJ}_k$ and all the branching random walks that are rooted at sibling nodes of ${\bfJ}_k$. Intuitively, since under $\widetilde{P}$ only the spine has positive drift, the probability that any path that coalesces with the spine outside of $\mathcal{T}_k^{(m)}$ has a very small chance of ever reaching level $t$ for sufficiently large $t$ and $m$.  The function $h_{m}$ will be used to quantify how rare this event is.

\begin{lemma} \label{L.h_m}
Under Assumption~\ref{Ass1}, the function 
$$h_m(x) =  \widetilde{E}\left[ 1(Z_0 > x) e^{-\alpha(V_{m+1}-x)^+} \right]$$
is d.R.i.~on $\mathbb{R}$ for any $m \geq 2$, and satisfies
$$\int_{-\infty}^\infty h_{m}(x) \, dx = \widetilde{E}\left[ \frac{e^{-\alpha (V_{m+1}-Z_0)^+}}{\alpha} + (Z_0 - V_{m+1})^+ \right]  < \infty.$$
\end{lemma}

\begin{proof}
We start by showing that $h_{m}$ is integrable, for which we note that
\begin{align*}
\int_{-\infty}^\infty h_{m}(x) \, dx &= \widetilde{E}\left[ \int_{-\infty}^{Z_0} e^{-\alpha (V_{m+1} - x)^+} dx \right] = \widetilde{E}\left[  \int_{V_{m+1} -Z_0}^\infty e^{-\alpha y^+} dy \right]  \\
&= \widetilde{E}\left[  \int_{V_{m+1} -Z_0}^{(V_{m+1}-Z_0)^+}   dy + \int_{(V_{m+1}-Z_0)^+}^\infty  e^{-\alpha y} dy \right] \\
&= \widetilde{E}\left[  (Z_0 - V_{m+1})^+  +  \frac{e^{-\alpha (V_{m+1}-Z_0)^+} }{\alpha}  \right] .
\end{align*} 
To see that $\widetilde{E}\left[ (Z_0-V_{m+1})^+ \right]$ is finite first note that
\begin{align*}
 \widetilde{E} \left[  (Z_0 - V_{m+1})^+ \right] &\leq \widetilde{E}\left[  Z_0^+ + (-V_{m+1})^+ \right] \leq \widetilde{E}\left[ Z_0^+ \right]  + (m+1) \widetilde{E} \left[ (-V_1)^+ \right] .
\end{align*}
By Lemma~\ref{L.DistOnJ} we have $\widetilde{E}\left[ (-V_1)^+ \right] = \widetilde{E}\left[ \hat X_1^- \right] < \infty$. For $\widetilde{E}\left[ Z_0^+ \right]$, recall $Q = e^{\xi_0}$ and write   
\begin{align}
\widetilde{E}\left[ Z_0^+ \right] 
&= \int_0^\infty \widetilde{E}\left[ \widetilde{P}\left(\left. \xi_0 \vee \bigvee_{i \neq {\bfJ}_1} (\log C_i + W_i) > t \right| \boldsymbol{\psi}_\emptyset \right) \right] dt \notag \\
&= \int_0^\infty\widetilde{E}\left[  1(\xi_0 > t) + 1(\xi_0 \leq t) \widetilde P\left(\left. \bigvee_{i\ne {\bfJ}_1} 1(\log C_i + W_i > t) \right| \bm{\psi}_\emptyset \right) \right] dt \notag \\
&\leq  \int_0^\infty\widetilde{E}\left[  1(\xi_0 > t) + 1\left(Q^\alpha  \leq e^{\alpha t} < \sum_{i \neq {\bfJ}_1} C_i^\alpha  \right)   \right] dt \notag \\
&\hspace{5mm} + \int_0^\infty\widetilde{E}\left[   1\left( Q^\alpha \vee \sum_{i \neq {\bfJ}_1} C_i^\alpha \leq e^{\alpha t} \right) \sum_{i \neq {\bfJ}_1} \widetilde{P}\left(\left. \log C_i + W_i > t \right| C_i \right)  \right] dt \notag \\
&= \widetilde{E}\left[ \xi_0^+  + \left( \frac{1}{\alpha} \log \left( \sum_{i \neq {\bfJ}_1}  C_i^\alpha \right) - \xi^+_0 \right)^+   \right] \label{eq:MomentCondition} \\
&\hspace{5mm} + \int_0^\infty\widetilde{E}\left[   1\left(Q^\alpha \vee \sum_{i \neq {\bfJ}_1} C_i^\alpha \leq e^{\alpha t} \right) \sum_{i \neq {\bfJ}_1} \overline{F}(t-\log C_i) \right] dt, \label{eq:SubtreeMax}
\end{align}
where $\overline{F}(x) = P(W > x)$ and $\eqref{eq:MomentCondition}$ is equal to $\widetilde{E}\left[ \frac{1}{\alpha} \log^+\left( Q^\alpha \vee \sum_{i \neq {\bfJ}_1} C_i^\alpha \right) \right]$.  By Corollary~\ref{L.unifbnd}, there exists a constant $B < \infty$ such that $\overline{F}(x) \leq B e^{-\alpha x}$ for $x \geq 0$. It follows that \eqref{eq:SubtreeMax} is bounded from above by
\begin{align*}
&B \int_0^\infty\widetilde{E}\left[   1\left(Q^\alpha \vee \sum_{i \neq {\bfJ}_1} C_i^\alpha \leq e^{\alpha t} \right) \sum_{i \neq {\bfJ}_1} e^{-\alpha(t-\log C_i)}  \right] dt \\
&=B\widetilde{E}\left[  \sum_{i \neq {\bfJ}_1} C_i^\alpha \int_{\frac{1}{\alpha} \log^+\left(Q^\alpha \vee \sum_{i \neq {\bfJ}_1} C_i^\alpha\right) }^\infty   e^{-\alpha t} dt \right] \\
&= \frac{B}{\alpha}  \widetilde{E}\left[  \sum_{i \neq {\bfJ}_1} C_i^\alpha e^{-\log^+\left(Q^\alpha \vee \sum_{i \neq {\bfJ}_1} C_i^\alpha\right) } \right] \leq \frac{B}{\alpha}.
\end{align*}
It follows that
\begin{align*}
\widetilde{E}\left[ Z_0^+ \right] &\leq \widetilde{E}\left[  \frac{1}{\alpha} \log^+\left( Q^\alpha \vee \sum_{i\neq {\bfJ}_1} C_i^\alpha \right) \right] + \frac{B}{\alpha} \\
&\leq \frac{1}{\alpha} E\left[ \sum_{i=1}^N C_i^\alpha   \log^+ \left( Q^\alpha \vee \sum_{i=1}^N C_i^\alpha \right)  \right] + \frac{B}{\alpha}  < \infty, 
\end{align*}
with finiteness provided by Assumption~\ref{Ass1}(e). 

It remains to show that $h_{m}$ is d.R.i., for which we note that for any $h > 0$, 
\begin{align*}
&\sum_{n = -\infty}^\infty \sup_{y \in (nh,(n+1)h]} h_{m}(y) \\
&\leq \sum_{n=-\infty}^\infty  \widetilde{E}\left[ 1(Z_0 > nh) e^{-\alpha(V_{m+1} - (n+1)h)^+} \right]  \\
&\leq \sum_{n=-\infty}^\infty \int_{(n-1)h}^{nh} \widetilde{E}\left[ 1(Z_0 > x) e^{-\alpha(V_{m+1} - x - 2h)^+} \right] dx \\
&= \widetilde{E}\left[  (Z_0 - V_{m+1} + 2h)^+ + \frac{e^{-\alpha(V_{m+1} - 2h - Z_0)^+}}{\alpha} \right] < \infty,
\end{align*}
so by Proposition~4.1(ii) in Chapter V of \cite{Asm2003}, $h_{m}$ is d.R.i.
\hfill\end{proof}

We are now ready to state the expression to which we will apply the Markov renewal theorem.

\begin{lemma} \label{L.Krenewal}
For any $t \in \mathbb{R}$ and $m \geq 2$ we have
\begin{align*}
\widetilde{E}\left[ K\left( M_m^{(m)}, t \right)  \right]  &\leq \sum_{k=0}^\infty \widetilde{E}\left[ g\left(M_{m+k}^{(m)}, t-V_k\right) \right], \\
\widetilde{E}\left[ K\left( M_m^{(m)}, t \right)  \right]  &\geq  \sum_{k=0}^\infty \widetilde{E}\left[ g\left(M_{m+k}^{(m)}, t-V_k\right) \right] - B \sum_{k=0}^\infty \widetilde{E}\left[ h_{m}(t - V_k) \right],
\end{align*}
where $B < \infty$ is the constant from Corollary~\ref{L.unifbnd}, 
\begin{equation}
g\left(M_m^{(m)}, t \right) = 1\left( \max_{{\bfi} \prec {\bfJ}_{m-1}} S_{\bfi} + Y_{\bfi} \leq t < V_{m-1} + \xi_{m-1} \right) e^{-\alpha(V_{m-1} - t)} D_{{\bfJ}_{m-1}}^{-1}, \label{eq:gFunction} 
\end{equation}
and
$$h_{m}(x) = \widetilde{E}\left[ 1(Z_{0} > x)  e^{-\alpha (V_{m+1} - x)^+}  \right].$$
\end{lemma} 

\begin{proof}
Start by  noting that
\begin{align*}
K\left(M_m^{(m)}, t\right) &= 1(|\gamma(t)| = m-1, \gamma(t) = {\bfJ}_{\tau(t)} ) e^{-\alpha(V_{m-1} - t)} D_{{\bfJ}_{m-1}}^{-1} \\
&\hspace{5mm} +  \widetilde{E}\left[  \left. 1(|\gamma(t)| \geq m, \gamma(t) = {\bfJ}_{\tau(t)} ) e^{-\alpha(V_{\tau(t)} - t)} D_{{\bfJ}_{\tau(t)}}^{-1} \right| M_m^{(m)} \right] \\
&= g\left( M_m^{(m)}, t \right) \\
&\hspace{5mm} + \widetilde{E}\left[ \left.  1\left( \Delta_{t,\tau(t)} \right)  1(|\gamma(t)| \geq m, \gamma(t) = {\bfJ}_{\tau(t)} ) e^{-\alpha(V_{\tau(t)} - t)} D_{{\bfJ}_{\tau(t)}}^{-1}  \right| M_m^{(m)} \right],
\end{align*}
where the $\Delta_{t,\tau(t)}$ is the event that $\xi_0 \leq t$ and no subtree rooted at any of the sibling nodes of ${\bfJ}_1$ reaches level $t$ before the spine does. Simply ignoring the indicator $1(\Delta_{t,\tau(t)} )$ and regenerating at node ${\bfJ}_1$ yields the inequality
\begin{align*}
\widetilde{E}\left[ K\left(M_m^{(m)}, t\right) \right] &\leq \widetilde{E}\left[ g\left( M_m^{(m)}, t \right) \right] + \widetilde{E}\left[ K\left(M_{m+1}^{(m)}, t-V_1\right) \right] \\
&\leq \sum_{k=0}^{n-1} \widetilde{E}\left[ g\left( M_m^{(m)}, t-V_k \right) \right] + \widetilde{E}\left[ K\left(M_{m+n}^{(m)}, t-V_n\right) \right] .
\end{align*}
To further bound the last expectation, let $s \in \mathbb{R}$, and note that on $\{\gamma(s) = {\bfJ}_{\tau(s)}\}$, no random walk on a path other than the chosen one reaches level $s$ before $V_k + \xi_k$. Hence by ignoring these branches up to level $m$ and restarting the branching process at ${\bfJ}_m$ with initial value $V_m$, we have that 
\begin{align*}
K\left(M_m^{(m)}, s\right) &= \widetilde{E}\left[  \left. 1(|\gamma(s)| \geq m-1, \gamma(s) = {\bfJ}_{\tau(s)} ) e^{-\alpha(V_{\tau(s)} - s)} D_{{\bfJ}_{\tau(s)}}^{-1} \right| M_m^{(m)} \right] \\
&\leq \widetilde E\left[\left.  1(\gamma(s - V_m) = {\bfJ}_{\tau(s-V_m)}) e^{-\alpha (V_{\tau(s - V_m)} - (s - V_m))} \right| V_m \right] \\
&=  e^{\alpha(s-V_m)} \overline{F}(s-V_m) ,
\end{align*}
where $\overline{F}(x) = P(W > x)$ and we used \eqref{eq:ISEstimator}. Moreover, by Corollary~\ref{L.unifbnd} we have that $e^{\alpha(s-V_m)} \overline{F}(s-V_m ) \leq B e^{-\alpha (V_m-s)^+}$, so we obtain for any $s \in \mathbb{R}$,
\begin{equation} \label{eq:K(M,s)bound}
K\left(M_m^{(m)}, s\right) \leq  B e^{-\alpha (V_m-s)^+}.
\end{equation}
Now replace $V_{m}$ with $V_{m+n} - V_n $ and $s = t-V_n$ in \eqref{eq:K(M,s)bound} to obtain that
$$\widetilde{E}\left[ K\left(M_{m+n}^{(m)}, t-V_n\right) \right] \leq B \widetilde E\left[e^{-\alpha(V_{m+n} - t)^+}\right]. $$

To obtain a lower bound note that $\Delta_{t,\tau(t)}^c \subseteq \{ Z_0 > t\}$, since the event $\{Z_0 > t\}$ states that either $\xi_0 > t$ or at least one of the subtrees rooted at a sibling node of ${\bfJ}_1$ reaches level $t$ at some point (even if this happens after the spine does). Hence, we obtain the following lower bound: 
\begin{align*}
&\widetilde{E}\left[ K\left(M_m^{(m)}, t\right) \right] \\
&\geq \widetilde{E}\left[ g\left( M_m^{(m)}, t \right) \right] + \widetilde{E}\left[ K\left(M_{m+1}^{(m)}, t-V_1\right) \right] - \widetilde{E}\left[ 1(Z_0 > t) K\left(M_{m+1}^{(m)}, t-V_1\right) \right] \\
&\geq  \widetilde{E}\left[ g\left(M_m^{(m)}, t\right)  \right] + \widetilde{E}\left[ g\left(M_{m+1}^{(m)}, t- V_1\right)  \right] + \widetilde{E}\left[  K\left(M_{m+2}^{(m)}, t-V_2 \right) \right]   \\
&\hspace{5mm}  -  \widetilde{E}\left[ 1(Z_{1} > t-V_1) K\left(M_{m+2}^{(m)}, t-V_2 \right) \right]  - \widetilde{E}\left[ 1(Z_{0} > t) K\left(M_{m+1}^{(m)}, t-V_1 \right) \right]   \\
&\geq \sum_{k=0}^{n-1} \widetilde{E}\left[ g\left(M_{m+k}^{(m)}, t- V_k \right)  \right] + \widetilde{E}\left[ K\left(M_{m+n}^{(m)}, t- V_n\right)  \right]   \\
&\hspace{5mm} - \sum_{k=0}^{n-1} \widetilde{E}\left[ 1(Z_{k} > t - V_k)  K\left(M_{m+k+1}^{(m)}, t- V_{k+1}\right)  \right].
\end{align*}
To provide a bound for the last sum, note that by replacing $V_m$ with $V_{m+k+1}- V_{k+1}$ and $s = t-V_{k+1}$ in \eqref{eq:K(M,s)bound} we obtain
\begin{align*}
\widetilde{E}\left[ 1(Z_{k} > t - V_k)  K\left(M_{m+k+1}^{(m)}, t- V_{k+1}\right)  \right] &\leq \widetilde{E}\left[ 1(Z_{k} > t - V_k)  B e^{-\alpha(V_{m+k+1} - t)^+}  \right] \\
&= B \widetilde{E}\left[ h_m(t-V_k) \right],
\end{align*}
where $h_m(x) =  \widetilde{E}\left[ 1(Z_0 > x) e^{-\alpha(V_{m+1}-x)^+} \right]$.

We have thus shown that for any $n \geq 1$
\begin{align*}
- B \sum_{k=0}^{n-1} \widetilde{E} \left[ h_m(t-V_k) \right] &\leq \widetilde{E}\left[ K\left(M_m^{(m)}, t\right) \right] - \sum_{k=0}^{n-1} \widetilde{E}\left[ g\left(M_{m+k}^{(m)}, t- V_k \right)  \right]  \\
 &\leq B \widetilde{E}\left[ e^{-\alpha(V_{m+n}-t)^+}  \right] .
\end{align*}
Monotone convergence and the observation that $V_n \to \infty$ $\widetilde{P}$-a.s.~immediately yields
$$\widetilde{E}\left[ K\left(M_m^{(m)}, t\right) \right]  \leq \sum_{k=0}^{\infty} \widetilde{E}\left[ g\left(M_{m+k}^{(m)}, t- V_k \right)  \right].$$
To obtain the lower bound note that by Lemma~\ref{L.h_m} we have that $h_{m}$ is nonnegative and d.R.i.~on $\mathbb{R}$, and therefore, $\sum_{k=0}^\infty \widetilde{E} \left[ h_{m}(t-V_k) \right] < \infty$ for all $t \in \mathbb{R}$. It follows that
$$\widetilde{E}\left[ K\left(M_m^{(m)}, t\right) \right] \geq \sum_{k=0}^{\infty} \widetilde{E}\left[ g\left(M_{m+k}^{(m)}, t- V_k \right)  \right] - B \sum_{k=0}^\infty \widetilde{E} \left[ h_{m}(t-V_k) \right].$$
This completes the proof.
\hfill\end{proof}


In order to connect our framework with the notation in the Markov renewal theorem from \cite{Alsmeyer_94}, recall that the
$$\hat X_n = V_n - V_{n-1}, \qquad n \geq 1,$$
define the increments of the random walk along the spine, and note that $\{(M_{m+n}^{(m)}, \hat X_{n}) : n \ge 0 \}$ is a time-homogeneous Markov process that only depends on the past through $\{ M_{m+n}^{(m)}: n \geq 0\}$. Hence, we can define a transition kernel $\mathbf{P}$ according to
\[
	 \mathbf{P}\left(M^{(m)}_{m+n}, \, A\times B\right) := \widetilde{P}\left(M^{(m)}_{m+n+1} \in A, \, \hat X_{n+1} \in B \Big| M^{(m)}_{m+n}, \hat X_{n}\right)
\]
for any measurable sets $A \subseteq \mathcal{S}^{(m)}$ and $B \subseteq \mathbb{R}$. Thus, $\{(M^{(m)}_{m+n}, V_{n}) : n\ge0\}$ is a Markov random walk in the sense of \cite{Alsmeyer_94}. Furthermore, by the way the process $\{M^{(m)}_{m+n}: n\ge0\}$ was constructed, it is $m$th-order stationary, in the sense that for each $n \ge 0$ the law under $\widetilde{P}$ of $\left( M_{m+n}^{(m)}, \ldots, M_{2m+n}^{(m)}  \right)$ is the same as that of $ \left(  M_{m}^{(m)}, \ldots, M_{2m}^{(m)} \right)$,
from which it follows that the unique stationary distribution for the chain $\{M^{(m)}_{m+n}: n \geq 0\}$ is given by
\begin{equation}
	\eta_m(\cdot) := \widetilde{E}\left[ \frac{1}{m} \sum_{n=0}^{m-1} 1\left( M^{(m)}_{m+n} \in \cdot \right) \right] =  \widetilde{P} \left(M^{(m)}_m \in \cdot \right).
\end{equation}

The idea is now to use Theorem 2.1 in \cite{Alsmeyer_94}, which states that provided that
\begin{enumerate}[(i)]

\item $\{(M^{(m)}_{m+n}, V_{n}) : n\ge0\}$ is a non-arithmetic and Harris recurrent Markov random walk, and

\item $g : \mathcal{S}^{(m)}\times \mathbb{R} \to \mathbb{R}$ is a measurable function such that $g(M, \cdot)$ is Lebesgue-a.e.~continuous for $\eta_m$-a.e.~$M$, and $g$ is d.R.i. in the sense that
\[
	\int_{\mathcal{S}^{(m)}} \sum_{n=-\infty}^\infty \sup_{y \in (n, n+1]} |g(M,y)|\,\eta_m (dM) < \infty,
\]

\end{enumerate}
then, it will follow that
\[
	\lim_{t\to\infty} \widetilde E\left[ \sum_{n=0}^\infty g\left(M_{n+m}^{(m)}, t - V_{n}\right) \right] = \frac{1}{\widetilde E[\hat X_1]} \int_{\mathcal{S}^{(m)}} \int_{\mathbb{R}} g(M, x)\,dx\, \eta_m(dM).
\]

The non-arithmeticity of $\{(M_{m+n}^{(m)}, V_{n}): n \geq 0\}$ follows from the non-arithmeticity of $\{V_n: n \geq 1\}$, which is ensured by Assumption~\ref{Ass1}(d) (see Lemma~\ref{L.DistOnJ}). To see that $\{M^{(m)}_{m+n} : n\ge0\}$ is Harris recurrent note that by construction, $M_{m+n}^{(m)}$ is independent of $\{ M_{n+2m}^{(m)}, M_{n+2m+1}^{(m)}, \dots\}$ for all $n \geq 0$, and therefore, by letting $\mathbf{Q}$ denote the transition kernel of $\{M_{m+n}^{(m)}: n \geq 0\}$ and $\mathbf{Q}^r$ its corresponding $r$-step transition kernel, we have
$$\mathbf{Q}^m( M_{m+n}^{(m)}, A) := \widetilde{P}\left( \left. M_{n+2m}^{(m)} \in A  \right| M_{m+n}^{(m)} \right) = \widetilde{P}\left( M_m^{(m)} \in  A \right) = \eta_m(A),$$
which satisfies the definition of a Harris chain (see Chapter VII \S3 in \cite{Asm2003}) with regeneration set $R = \mathcal{S}^{(m)}$, probability measure $\lambda = \eta_m$ and $\epsilon = 1$.

The last ingredient before applying the Markov renewal theorem of \cite{Alsmeyer_94} to our situation is to show that the function $g$ defined by \eqref{eq:gFunction} satisfies the necessary conditions. The corresponding result is given by the following lemma. 

\begin{lemma} \label{L.gProperties}
Let $g:\mathcal{S}^{(m)} \times \mathbb{R} \to \mathbb{R}$ be defined by \eqref{eq:gFunction}. Then, under Assumption~\ref{Ass1}, $g(M,\cdot)$ is Lebesgue-a.e.~continuous for $\eta_m$-a.e.~$M$, and $g$ is d.R.i. Moreover, 
\begin{align*}
\int_{\mathcal{S}^{(m)}}  \int_{-\infty}^\infty g(M,x) \, dx \, \eta_m(dM) &=   \frac{1}{\alpha} \widetilde{E}\left[   D_{{\bfJ}_{m-1}}^{-1} \left( e^{\alpha \xi_{m-1}} - e^{\alpha\left( \max_{{\bfi}  \prec {\bfJ}_{m-1}} (S_{\bfi} + Y_{\bfi}) - V_{m-1} \right)} \right)^+  \right] \\
&< \infty.
\end{align*}
\end{lemma}

\begin{proof} 
We start by showing that $g(M,\cdot)$ is Lebesgue-a.e.~continuous. To see this, let $M \in \mathcal{S}^{(m)}$, identify its $m$ generations, all its weights, and its spine. Then note that 
$$g\left( M, t\right) = 1( a(M) \leq t \leq b(M) ) c(M) e^{\alpha t}$$
for some fixed numbers $a(M), b(M)$ and $c(M)$. Therefore, it is Lebesgue-a.e.~continuous. 

It remains to show that $g$ is d.R.i., for which we note that for any $M \in \mathcal{S}^{(m)}$ for which we have identified its generations, weights, and spine, we have
\begin{align*}
&\sum_{n=-\infty}^\infty \sup_{y \in (n, n+1]} |g(M,y)| \\
	&= \sum_{n=-\infty}^\infty \sup_{y \in (n, n+1]}  1 \left( \max_{ {\bfi} \prec {\bfJ}_{m-1} }  S_{\bfi} + Y_{\bfi} \leq y  < V_{m-1} + \xi_{m-1} \right) e^{-\alpha(V_{m-1}-y)} D_{{\bfJ}_{m-1}}^{-1}  \\
	&\leq D_{{\bfJ}_{m-1}}^{-1} e^{-\alpha V_{m-1}} \sum_{n=-\infty}^\infty  1\left( V_{m-1} + \xi_{m-1} > n, \max_{{\bfi} \prec {\bfJ}_{m-1}}   S_{\bfi} + Y_{\bfi} \leq n+1 \right)   e^{\alpha(n+1)}   \\
	&\leq D_{{\bfJ}_{m-1}}^{-1} e^{-\alpha V_{m-1}}  \sum_{n=-\infty}^\infty \int_{n+1}^{n+2}   1\left( V_{m-1} + \xi_{m-1} +2> x,  \max_{{\bfi} \prec {\bfJ}_{m-1}}   S_{\bfi} + Y_{\bfi} \leq x \right)    e^{\alpha x}  \, dx \\
	&=  D_{{\bfJ}_{m-1}}^{-1}  e^{-\alpha V_{m-1}}\int_{-\infty}^\infty 1 \left( \max_{ {\bfi} \prec {\bfJ}_{m-1} }  S_{\bfi} + Y_{\bfi} \leq x  < V_{m-1} + \xi_{m-1} +2 \right) e^{\alpha x}   \, dx \\
	&=   D_{{\bfJ}_{m-1}}^{-1}  e^{-\alpha V_{m-1}} \frac{1}{\alpha} \left( e^{\alpha(V_{m-1}+\xi_{m-1} + 2)} - e^{\alpha \left( \max_{ {\bfi} \prec {\bfJ}_{m-1} }  S_{\bfi} + Y_{\bfi} \right) }    \right)^+ \\
	&= \frac{1}{\alpha} D_{{\bfJ}_{m-1}}^{-1}   \left( e^{\alpha(\xi_{m-1} + 2)} - e^{\alpha \left( \max_{ {\bfi} \prec {\bfJ}_{m-1} }  S_{\bfi} + Y_{\bfi} - V_{m-1} \right) }    \right)^+.
\end{align*}
It follows that
\begin{align*}
&\int_{\mathcal{S}^{(m)}} \sum_{n=-\infty}^\infty \sup_{y \in (n, n+1]} |g(M,y)|  \, \eta_m(dM) \\
&\leq  \frac{1}{\alpha} \widetilde{E}\left[   D_{{\bfJ}_{m-1}}^{-1} \left( e^{\alpha ( \xi_{m-1} + 2)} - e^{\alpha\left( \max_{{\bfi}  \prec {\bfJ}_{m-1}} (S_{\bfi} + Y_{\bfi}) - V_{m-1} \right)} \right)^+  \right] \\
&\leq \frac{1}{\alpha} \widetilde{E}\left[   D_{{\bfJ}_{m-1}}^{-1} e^{\alpha ( \xi_{m-1} + 2)} \right] = \frac{e^{2\alpha}}{\alpha} E[ Q^\alpha] < \infty, 
\end{align*}
which implies that $g$ is d.R.i. 

To complete the proof, note that essentially the same steps followed above give that
$$\int_{\mathcal{S}^{(m)}}  \int_{-\infty}^\infty g(M,x) \, dx \, \eta_m(dM) =   \frac{1}{\alpha} \widetilde{E}\left[   D_{{\bfJ}_{m-1}}^{-1} \left( e^{\alpha \xi_{m-1}} - e^{\alpha\left( \max_{{\bfi}  \prec {\bfJ}_{m-1}} (S_{\bfi} + Y_{\bfi}) - V_{m-1} \right)} \right)^+  \right] $$
and that the right hand side is finite. 
\hfill\end{proof}

We are finally ready to prove Theorem~\ref{T.Main}.

\begin{proof}[Proof of Theorem~\ref{T.Main}]
From the derivations at the beginning of the subsection and Lemma~\ref{L.Krenewal} we have that for any $m \geq 2$, 
\begin{align*}
 \widetilde{E} \left[ 1(\gamma(t) = {\bfJ}_{\tau(t)}) e^{-\alpha(V_{\tau(t)} - t)} D_{{\bfJ}_{\tau(t)}}^{-1} \right] &\leq   \sum_{k=0}^{m-1} \widetilde{E} \left[ 1(|\gamma(t)| = k, \gamma(t) = {\bfJ}_{\tau(t)}) e^{-\alpha(V_k - t)} D_{{\bfJ}_k}^{-1} \right] \\
& \hspace{5mm} + \sum_{n=0}^\infty  \widetilde{E}\left[ g\left(M_{m+n}^{(m)}, t - V_n\right)  \right] .
 \end{align*}
  To see that each of the first $m$ expectations converges to zero as $t \to \infty$, note that
 \begin{align*}
  \widetilde{E} \left[ 1(|\gamma(t)| = k, \gamma(t) = {\bfJ}_{\tau(t)}) e^{-\alpha(V_k - t)} D_{{\bfJ}_k}^{-1} \right] &\leq  \widetilde{E} \left[ 1(V_k + \xi_k > t) e^{-\alpha(V_k - t)} D_{{\bfJ}_k}^{-1} \right] \\
  &=  \widetilde{E} \left[ u(t-V_k) \right] ,
 \end{align*}
 where $u(x) = e^{\alpha x} \widetilde{E}\left[ 1(\xi_0 > x) D_{{\bfJ}_0}^{-1}  \right] = e^{\alpha x} P(Y > x)$. Since $u$ is bounded and integrable on $(-\infty, \infty)$, it follows from the bounded convergence theorem that
 \begin{align*}
 \limsup_{t \to \infty}  \widetilde{E} \left[ 1(|\gamma(t)| = k, \gamma(t) = {\bfJ}_{\tau(t)}) e^{-\alpha(V_k - t)} D_{{\bfJ}_k}^{-1} \right] &\leq  \widetilde{E} \left[  \limsup_{t \to \infty} u(t-V_k) \right] = 0. 
 \end{align*}
 Now use the Markov renewal theorem (Theorem~2.1 in \cite{Alsmeyer_94}) and Lemma~\ref{L.gProperties} to obtain that 
 \begin{align*}
& \limsup_{t \to \infty}  \widetilde{E} \left[ 1(\gamma(t) = {\bfJ}_{\tau(t)}) e^{-\alpha(V_{\tau(t)} - t)} D_{{\bfJ}_{\tau(t)}}^{-1} \right] \\
&\leq \sum_{k=0}^{m-1} \lim_{t \to \infty} \widetilde{E} \left[ 1(|\gamma(t)| = k, \gamma(t) = {\bfJ}_{\tau(t)}) e^{-\alpha(V_k - t)} D_{{\bfJ}_k}^{-1} \right] \\
&\quad + \frac{1}{ \widetilde{E}[\hat X_1]}\int_{\mathcal{S}^{(m)}}  \int_{-\infty}^\infty g(M,x) \, dx \, \eta_m(dM) \\
 &= \frac{1}{\alpha \mu} \widetilde{E}\left[   D_{{\bfJ}_{m-1}}^{-1} \left( e^{\alpha \xi_{m-1}} - e^{\alpha\left( \max_{{\bfi}  \prec {\bfJ}_{m-1}} (S_{\bfi} + Y_{\bfi}) - V_{m-1} \right)} \right)^+  \right] =: H_{m-1} ,
 \end{align*}
where $\mu = \widetilde{E}[\hat X_1] = E\left[ \sum_{i=1}^N C_i^\alpha \log C_i \right] > 0$. 

To obtain a lower bound use Lemma~\ref{L.Krenewal} again to obtain that
$$ \widetilde{E} \left[ 1(\gamma(t) = {\bfJ}_{\tau(t)}) e^{-\alpha(V_{\tau(t)} - t)} D_{{\bfJ}_{\tau(t)}}^{-1} \right]  \geq \sum_{n=0}^\infty  \widetilde{E}\left[ g\left(M_{m+n}^{(m)}, t - V_n\right)  \right]  -  B \sum_{n=0}^\infty \widetilde{E}\left[ h_{m}(t-V_n) \right] .$$
Now use the two-sided renewal theorem (see Theorem~4.2 in \cite{Ath_McD_Ney_78}) and Lemma~\ref{L.h_m} to obtain that
\begin{align*}
\lim_{t \to \infty} \sum_{n=0}^\infty \widetilde{E}\left[ h_{m}(t-V_n) \right] &= \frac{1}{\mu} \int_{-\infty}^\infty h_{m}(x) \, dx \\
&= \frac{1}{\mu} \widetilde{E}\left[ \frac{e^{-\alpha(V_{m+1}-Z_0)^+}}{\alpha} + (Z_0 - V_{m+1})^+ \right]  < \infty.
\end{align*}
It follows that
\begin{align*}
\liminf_{t \to \infty}  \widetilde{E} \left[ 1(\gamma(t) = {\bfJ}_{\tau(t)}) e^{-\alpha(V_{\tau(t)} - t)} D_{{\bfJ}_{\tau(t)}}^{-1} \right]  &\geq H_{m-1} \\
&\quad -  \frac{B}{\mu} \widetilde{E}\left[ \frac{e^{-\alpha(V_{m+1}-Z_0)^+}}{\alpha} + (Z_0 - V_{m+1})^+ \right].
\end{align*}
Since $V_m \to \infty$ $\widetilde{P}$-a.s., we have that
$$\lim_{m \to \infty} \widetilde{E}\left[ \frac{e^{-\alpha(V_{m+1}-Z_0)^+}}{\alpha} + (Z_0 - V_{m+1})^+ \right] = 0, $$
and we conclude that
$$\lim_{t \to \infty} e^{\alpha t} P(W > t)  = \lim_{t \to \infty}  \widetilde{E} \left[ 1(\gamma(t) = {\bfJ}_{\tau(t)}) e^{-\alpha(V_{\tau(t)} - t)} D_{{\bfJ}_{\tau(t)}}^{-1} \right] = \lim_{m \to \infty} H_m =: H.$$
The positivity of $H$ under Assumption~\ref{Ass2} follows from representation \eqref{eq:EquivH} and Theorem~3.4 in \cite{Jel_Olv_15}. 
\hfill\end{proof}

\subsection{The importance sampling estimator $Z(t)$}

The last part of the paper contains the proofs of Lemmas \ref{L.Tovermu} and \ref{L.BddRelError} in Section~\ref{S.Simulation}. The first of these establishes the asymptotic behavior of $\tau(t)$ as $t \to \infty$, and the second one proves the strong efficiency of our proposed estimator.

\begin{proof}[Proof of Lemma~\ref{L.Tovermu}] Since $\tau(y)$ is monotone nondecreasing in $y$, $\lim_{y\to \infty} \tau(t) = \sup_y \tau(y)$ exists, and for any $k > 0$ and $x \in \mathbb{R}$, 
\[
	\widetilde P\left(\lim_{y\to\infty} \tau(y) > k \right) \ge \widetilde P(\tau(x) > k ) = \widetilde P\left(\max_{j \le k} V_j + \xi_j \le x\right). 
\]
Letting $x \to\infty$, we see that $\widetilde P(\lim_{y\to\infty} \tau(y) > k) = 1$, and since this is true for all $k$, $\lim_{y\to\infty} \tau(y) = \infty$ $\widetilde P$-a.s. It remains to show that $\tau(t)/t \to 1/\mu$ $\widetilde{P}$-a.s.~provided $\widetilde{E}\left[ \xi_0^+ \right] < \infty$ and $\widetilde{P}(\xi_0 > -\infty) > 0$, which are implied by Assumption~\ref{Ass1}.

Start by noting that for any $t > 0$, 
$$\frac{1}{\tau(t)+1} \max_{0 \leq k < \tau(t)} (V_k + \xi_k) \leq \frac{t}{\tau(t)+1} \leq \frac{V_{\tau(t)}+\xi_{\tau(t)} }{\tau(t)+1}.$$
To obtain an upper bound for $t/(\tau(t)+1)$ note that 
$$\limsup_{t \to \infty} \frac{t}{\tau(t)+1}  \leq \limsup_{n \to \infty} \frac{V_n + \xi_n^+}{n} \leq \mu + \limsup_{n \to \infty} \frac{\xi_n^+}{n} = \mu,$$
where $\limsup_{n \to \infty} \xi_n^+/n = 0$ $\widetilde{P}$-a.s.~since $\widetilde{E}[\xi_0^+] < \infty$ and $\limsup_{n \to \infty} V_n/n = \mu$ $\widetilde{P}$-a.s.~by the strong law of large numbers since $\widetilde{E}[|V_1|] < \infty$. To obtain a lower bound let $m_n = n/\log n$ and note that since $t/(\tau(t)+1) > 0$ we have
\begin{align*}
\liminf_{t \to \infty} \frac{t}{\tau(t)+1} &\geq \liminf_{n \to \infty} \left( \frac{1}{n} \max_{0 \leq k < n} (V_k + \xi_k)  \right)^+  \geq \liminf_{n \to \infty}  \frac{1}{n} \max_{m_n \leq k < n} (V_k + \xi_k)^+ \\
&\geq \liminf_{n \to \infty}  \frac{1}{n} \max_{m_n \leq k < n} \left( (\mu k + \xi_k)^+ - (\mu k - V_k)^+ \right) \\
&\geq \liminf_{n \to \infty} \frac{1}{n} \max_{m_n \leq k < n} (\mu k + \xi_k)^+ - \limsup_{n \to \infty} \frac{1}{n} \max_{m_n \leq k < n} (\mu k-V_k)^+ \\
&\geq \liminf_{n \to \infty} \frac{1}{n} \max_{m_n \leq k < n} (\mu k + \xi_k)^+ - \limsup_{n \to \infty}  \left(\mu - \frac{V_n}{n} \right)^+.
\end{align*}
Since the strong law of large numbers gives that $\limsup_{n \to \infty} (\mu - V_n/n)^+ = 0$ $\widetilde{P}$-a.s., it only remains to show that $\liminf_{n \to \infty} n^{-1} \max_{m_n \leq k < n} (\mu k+\xi_k)^+ \geq \mu$ $\widetilde{P}$-a.s. To show that this is indeed the case, fix $0 < \epsilon < \mu$, define $M_\epsilon = \lceil e^{(1-\epsilon/2)^{-1}} \rceil$, and note that
\begin{align*}
&\sum_{n = 3}^\infty \widetilde{P}\left( \frac{1}{n} \max_{m_n \leq k < n} (\mu k+ \xi_k)^+ - \mu < -\epsilon \right) \\
&= \sum_{n = 3}^\infty \widetilde{P}\left(  \max_{m_n \leq k < n} (\mu k+ \xi_k)^+ < (\mu -\epsilon)n \right) \\
&\leq M_\epsilon + \sum_{n=M_\epsilon +1}^\infty \widetilde{P}\left( \max_{ \lceil (1-\epsilon/2)n \rceil \leq k < n} (\mu k+\xi_k)^+ < (\mu-\epsilon)n \right) \\
&\leq M_\epsilon + \sum_{n= M_\epsilon +1}^\infty  \prod_{k= \lceil (1-\epsilon/2)n \rceil}^n \widetilde{P}\left( \mu k + \xi_k < (\mu-\epsilon) n \right)  \\
&\leq M_\epsilon + \sum_{n= M_\epsilon +1}^\infty  \prod_{k= \lceil (1-\epsilon/2)n \rceil}^n \widetilde{P}\left(  \xi_0 < -(\epsilon/2) n  \right) \\
&\leq  M_\epsilon + \sum_{n= M_\epsilon +1}^\infty  \widetilde{P}(2\xi_0 < -\epsilon n)^{(\epsilon/2)n -1} .
\end{align*}
Since by assumption we have that $\widetilde{P}(\xi_0 > -\infty) > 0$, then there exists $n_0 > M_\epsilon$ such that $\widetilde{P}(2\xi_0 < -\epsilon n) < 1$ for all $n \geq n_0$, which shows that the series above converges. Finally, use the Borel-Cantelli lemma to conclude that
$$\liminf_{n \to \infty} \frac{1}{n} \max_{m_n \leq k < n} (\mu k+\xi_k)^+ - \mu = 0 \qquad \widetilde{P}\text{-a.s.}, $$
which in turn implies that
$$\lim_{t \to \infty} \frac{t}{\tau(t)} = \mu \qquad \widetilde{P}\text{-a.s.}$$
\hfill\end{proof}

\begin{proof}[Proof of Lemma~\ref{L.BddRelError}] From Theorem~\ref{T.Main}, 
\[
	P(W > t)^2 \sim H^2 e^{-2\alpha t} \quad\mbox{as}\quad t\to\infty
\]
for $H^2 > 0$, and so
\begin{align*}
	\limsup_{t\to\infty} \frac{\widetilde\var(Z(t))}{P(W > t)^2} &\le \limsup_{t\to\infty} \frac{\widetilde{E}\left[1({\bfJ}_{\tau(t)} = \gamma(t)) e^{-2\alpha V_{\tau(t)}}D^{-2}_{{\bfJ}_{\tau(t)}}\right]}{H^2 e^{-2\alpha t} }\\
	&\le H^{-2} \limsup_{t \to \infty} \widetilde E\left[ e^{-2\alpha (V_{\tau(t)}-t)}D^{-2}_{{\bfJ}_{\tau(t)}} \right].
\end{align*}
Now, by an argument analogous to that in the proof of Lemma~\ref{L.Overshoot}, we have that 
\begin{align*}
	 \widetilde E\left[ e^{-2\alpha(V_{\tau(t)} - t)}D^{-2}_{{\bfJ}_{\tau(t)}} \right] 
	&\le \sum_{n=0}^\infty \widetilde E\left[ v(t - V_n) \right], 
\end{align*}
where $v(x) = e^{2\alpha x} E[1(Y > x) D^{-1}]$, which is integrable since 
\begin{align*}
	\int_{-\infty}^\infty v(x)\,dx &= \int_{-\infty}^\infty e^{2\alpha x} E\left[1(Y > x) D^{-1}\right]\,dx \\
	&= E\left[ \int_{-\infty}^Y e^{2\alpha x}D^{-1}\,dx \right] = \frac{1}{2\alpha} E\left[Q^{2\alpha}D^{-1}\right] < \infty. 
\end{align*}
Then, $v$ is d.R.i. by the same argument as in the proof of Lemma~\ref{L.Overshoot} for the function $u$, and hence 
\[
	\limsup_{t\to\infty} \widetilde E\left[ e^{-2\alpha(V_{\tau(t)} - t)} D^{-2}_{{\bfJ}_{\tau(t)}} \right] \le \limsup_{t\to\infty} \sum_{n=0}^\infty \widetilde E\left[ v(t - V_n) \right] = \int_{-\infty}^\infty v(x)\,dx < \infty. 
\]
The proof is the same for the estimator $1({\bfJ}_{\tau(t)} = \gamma(t)) e^{-\alpha V_{\tau(t)}}$ in the case of independent $Q$.
\hfill\end{proof}

\subsection{The bounded perturbations, non-branching case}

We end the paper with a short proof of Theorem~1 in \cite{Araman} for the non-branching case $N \equiv 1$ and bounded $Q$, which establishes the exponential asymptotic behavior of $P(W > x)$. As mentioned earlier, a similar approach could be used for the branching case with bounded $Q$, however, since our goal was not to establish the exponential asymptotic itself (for which the implicit renewal theorem on trees in \cite{Jel_Olv_15} can be used), but rather shed some light into the event leading to the constant, we do not pursue this idea any further. 
In the $N \equiv 1$ case, 
$W$ is the maximum of
a negative drift perturbed random walk, that is
\[W\stackrel{\mathcal{D}}{=} \sup_{n\ge 0} \left(V_n+\xi_n\right) .\]

\begin{theo}[Theorem 1 in \cite{Araman}] \label{T.BoundedQNonBranching}
Suppose that $N \equiv 1$ and one of the following holds:
\begin{enumerate}[(a)]
\item $\{(\xi_{i-1}, \hat X_i): i \geq 1\}$ are i.i.d., or,
\item $\{\xi_i: i \geq 0\}$ is a stationary sequence, independent of the i.i.d.~sequence $\{ \hat X_i: i \geq 1\}$.
\end{enumerate}
In either case, assume that $P(\xi_0 \leq c) = 1$ for some constant $c$ and $P(\xi_0 > 0) > 0$; we allow the possibility that $P(\xi_0 = -\infty) > 0$ but assume that $P( \hat X_1 = -\infty) = 0$. Assume further that $E[e^{\alpha \hat X_1}] = 1$ and $E[\hat X_1 e^{\alpha \hat X_1}] \in (0, \infty)$ for some $\alpha > 0$, and that the measure 
$P(\hat X_1 \in dx)$ is non-arithmetic. 
Then, 
$$P(W > x) \sim H e^{-\alpha x}, \qquad x \to \infty,$$
for some constant $0 < H < \infty$.  
\end{theo}

\begin{proof}
Define the filtration $\mathcal{H}_n = \sigma( \hat X_i: 1 \leq i \leq n)$ for $n \geq 0$ and $\mathcal{H}_0 = \sigma(\varnothing)$. Let $T(x) = \inf\{ n \geq 1: V_n > x\}$ and note that it is a stopping time with respect to $\{ \mathcal{H}_n: n \geq 0\}$. Also let $\tau(x) = \inf\{n \ge 1 : V_n + \xi_n > x\}$. Since the perturbations are bounded, we have $\tau(x) \geq T(x-c)$, and since the drift of $V_n = \hat X_1 + \dots + \hat X_n$ is positive under $\widetilde{P}$, then $\widetilde{P}(T(t) < \infty) = 1$ for all $t \geq 0$. Now let $y = x-c$ and write,
\begin{align*}
& P(W > x) \\
&=  P( T(y) \leq \tau(x) < \infty ) \\
&= P\left( T(y) < \infty, \,   \sup_{k \geq 0}\, \left(  V_{T(y)-1} + V_{T(y)+k} + \xi_{T(y)+k}\right) > x \right) \\
&= E\left[ 1(T(y) < \infty) E\left[ \left. 1\left(V_{T(y)} + \max\left\{ \xi_{T(y)}, \, \sup_{k \geq 1} \left( V_{T(y)+k} + \xi_{T(y)+k} \right) \right\} > x \right) \right| \mathcal{H}_{T(y)} \right] \right]. 
\end{align*}
Since $ \max\left\{ \xi_{T(y)}, \sup_{k \geq 1} ( V_{T(y)+k} + \xi_{T(y)+k} ) \right\}$ is independent of $\mathcal{H}_{T(y)}$ and has the same distribution as $W$, we have that
$$E\left[ \left. 1\left(V_{T(y)} + \max\left\{ \xi_{T(y)}, \, \sup_{k \geq 1} \left( V_{T(y)+k} + \xi_{T(y)+k} \right) \right\} > x \right) \right| \mathcal{H}_{T(y)} \right] = \overline{F}(x-V_{T(y)}),$$
where $\overline{F}(t) = P(W > t)$. Hence,
\begin{align*}
e^{\alpha x} P(W> x) &= e^{\alpha x} E\left[ 1(T(y) < \infty) \overline{F}(x - V_{T(y)}) \right] \\
&= e^{\alpha x} E\left[ 1(T(y) < \infty) \overline{F}(x - V_{T(y)}) e^{-\alpha V_{T(y)}} L_{T(y)} \right] \\
&=  e^{\alpha x} \widetilde{E}\left[ 1(T(y) < \infty)  \overline{F}(x - V_{T(y)}) e^{-\alpha V_{T(y)}} \right]  \\
&= \widetilde{E}\left[ \overline{F}(x - V_{T(y)}) e^{-\alpha (V_{T(y)} - x)} \right] \\
&= \widetilde{E}\left[ \overline{F}(c - B(x-c)) e^{-\alpha B(x-c) } \right] e^{\alpha c},
\end{align*}
where $B(t) = V_{T(t)} - t \geq 0$
is the overshoot process of the random walk $\{V_n: n \geq 1\}$. 
Since $P(\hat X_1 \in dx)$ is non-arithmetic, so is $\widetilde P(\hat X_1 \in dx) = E[1(\hat X_1 \in dx) e^{\alpha \hat X_1}]$, and hence 
by Theorem 2.1 in Chapter VIII of \cite{Asm2003}, $B(t)$ converges in $\widetilde{P}$-distribution as $t \to \infty$ to an a.s. finite limit $B(\infty)$, and therefore, 
$$\lim_{x \to \infty} \widetilde{E}\left[ \overline{F}(c - B(x-c)) e^{-\alpha B(x-c) } \right] e^{\alpha c}  = \widetilde{E}\left[ \overline{F}(c - B(\infty)) e^{-\alpha B(\infty) } \right]  =: H.$$
To see that $H > 0$, note that $\overline{F}(t) \geq \sup_{n \geq 0} P(V_n + \xi_n > t, \xi_n > 0) \geq \sup_{n \geq 0} P(V_n > t) P(\xi_0 > 0)$. The condition $E[\hat X_1 e^{\alpha \hat X_1}] > 0$ implies that $P(\hat X_1 > 0) > 0$, and in particular there must be some $\epsilon > 0$ such that $P(\hat X_1 > \epsilon) > 0$. For any $t > 0$, if $k > t/\epsilon$, then 
\[
	P(V_k > t) \ge P \left( \hat X_1 > \frac{t}{k}, \ldots, \hat X_k > \frac{t}{k} \right) \ge P(\hat X_1 > \epsilon)^k > 0. 
\]
Since $P(\xi_0 > 0) > 0$ as well, we the have $\overline{F}(t) > 0$ for each $t>0$. In particular, since $B(\infty) \ge 0$ a.s., $\overline{F}(c - B(\infty)) \ge \overline{F}(c) > 0$. Therefore, $B(\infty) < \infty$ a.s. and $\overline{F}(c-B(\infty)) > 0$ a.s.
imply together that $H > 0$. 
\end{proof}

\bibliographystyle{plain}

\end{document}